\documentclass[11pt,twoside]{amsart}
\usepackage{amsmath}
\usepackage{amsthm}
\usepackage{aliascnt}
\usepackage{amsfonts}
\usepackage{amssymb}
\usepackage{dutchcal}
\usepackage{mathrsfs}
\usepackage{mathtools}
\usepackage{upgreek}
\DeclareMathAlphabet{\mathcal}{OMS}{cmsy}{m}{n}

\usepackage{enumitem}
\usepackage{indentfirst}
\usepackage{ulem}
\usepackage{tikz}
\usepackage{tikz-cd}	
\usepackage{array}  
\newcolumntype{L}{>{$}l<{$}} 
\usepackage{verbatim}
\usepackage[colorlinks=true,
linkcolor=black,
citecolor=black,
urlcolor=black]{hyperref}
\usepackage[capitalize]{cleveref}
\makeatletter
\renewcommand{\eqref}[1]{\textup{(\ref{#1})}}
\makeatother

\usepackage{lipsum}
\usepackage{comment}
\usepackage{adjustbox}
\usepackage{scalerel}
\usepackage{marginnote}
\usepackage[toc]{appendix} 
\usepackage[style = numeric, doi=false,isbn=false,url=false, giveninits=true, maxalphanames=10,maxbibnames=99]{biblatex}
\usepackage{microtype}
\usepackage{xcolor}

\setlist[enumerate]{
	leftmargin=*,
	labelsep=0.5em
}

\usepackage[]{todonotes}

\def \ix {I_{\scaleto{X}{4pt}}}
\def \zs {Z_{\scaleto{S}{4pt}}}
\def \lx {\mathbb{L}_{\scaleto{X}{4pt}}}
\def \ly {\mathbb{L}_{\scaleto{Y}{4pt}}}
\def \hx {h_{\scaleto{X}{4pt}}}
\def \hy {h_{\scaleto{Y}{4pt}}}
\def \gx {g_{\scaleto{X}{4pt}}}
\def \gy {g_{\scaleto{Y}{4pt}}}
\def \tmcc {\widetilde{\mathcal{C}}}
\def \tcl {\widetilde{C}_l}
\def \tpl {\widetilde{P}_l}
\def \lhom {\textup{hom}}

\def \sing {\textup{Sing} }
\def \ts {\widetilde{S}}

\def \mcs {\mathcal{S}}
\def \mcy {\mathcal{Y}}
\def \mcz {\mathcal{Z}}
\def \mcc {\mathcal{C}}
\def \mbc {\mathbb{C}}
\def \mbp {\mathbb{P}}
\def \mso {\mathcal{O}}
\def \mcf {\mathcal{F}}
\def \mbq {\mathbb{Q}}
\def \mbp {\mathbb{P}}

\def \mbz {\mathbb{Z}}

\def \pic {\textup{Pic}}

\def \ch {\textup{CH}}
\def \ord {\textup{ord}}
\def \ra {\rightarrow}

\renewbibmacro{in:}{} 
\DeclareFieldFormat{pages}{#1}

\newtheorem{thm}{Theorem}[section]

\newaliascnt{prop}{thm}
\newtheorem{prop}[prop]{Proposition}
\aliascntresetthe{prop}

\newaliascnt{lem}{thm}
\newtheorem{lem}[lem]{Lemma}
\aliascntresetthe{lem}

\newaliascnt{cor}{thm}
\newtheorem{cor}[cor]{Corollary}
\aliascntresetthe{cor}

\theoremstyle{definition}
\newaliascnt{defn}{thm}
\newtheorem{defn}[defn]{Definition}
\aliascntresetthe{defn}

\newaliascnt{rmk}{thm}
\newtheorem{rmk}[rmk]{Remark}
\aliascntresetthe{rmk}

\newaliascnt{ex}{thm}
\newtheorem{ex}[ex]{Example}
\aliascntresetthe{ex}

\newaliascnt{conj}{thm}

\aliascntresetthe{conj}

\newaliascnt{quest}{thm}
\newtheorem{quest}[quest]{Question}
\aliascntresetthe{quest}
\Crefname{thm}{Theorem}{Theorems}
\Crefname{prop}{Proposition}{Propositions}
\Crefname{lem}{Lemma}{Lemmas}
\Crefname{cor}{Corollary}{Corollaries}
\Crefname{defn}{Definition}{Definitions}
\Crefname{rmk}{Remark}{Remarks}
\Crefname{ex}{Example}{Examples}
\Crefname{conj}{Conjecture}{Conjectures}
\Crefname{quest}{Question}{Questions}

\title{Constant Cycle Surfaces on Fano Varieties of Cubic Fourfolds}
\author{Jiexiang Huang} 
\thanks{The author is supported by the ERC Synergy Grant HyperK (ID 854361).}
\hypersetup{colorlinks = true, linkcolor=black}

\begin{document}
	\begin{abstract}
	Huybrechts proved the finiteness  of constant cycle curves of  fixed order in any linear system $|L|$ on a K3 surface. In this paper, we  study constant cycle surfaces on the Fano variety of lines $F(X)$ of a smooth cubic fourfold $X$. Fano surfaces $F(Y) \subset F(X)$ of hyperplane sections $Y \subset X$ are higher-dimensional analogues of curves on K3 surfaces. We prove that there are at most finitely many constant cycle surfaces of the form $F(Y)$ of any fixed order on $F(X)$.
	\end{abstract}
\maketitle
\tableofcontents

\section{Introduction}

Constant cycle subvarieties on a complex hyperk\"ahler variety $Z$  are  subvarieties whose points are all rationally equivalent in $Z$. Such subvarieties may be viewed as   Chow-theoretic generalizations of rational subvarieties, and are particularly interesting because the Chow group  $\ch_0(Z)$ of zero-cycles is infinite-dimensional \cite{Mumford69surface, Roitman1971gamma}. A constant cycle subvariety $\Gamma \subset Z$ has dimension at most  $\frac{1}{2}\dim(Z)$; if equality holds, the subvariety $\Gamma$ is Lagrangian, and  $Z$ contains at most countably many such subvarieties \cite{Voisin2016}. 

\subsection{Background and Motivation} K3 surfaces are the simplest hyperk{\"a}hler varieties. Constant cycle curves on K3 surfaces have been studied systematically  by Huybrechts \cite{Huybrechts2014}. Typical examples include rational curves and fixed curves of non-symplectic automorphisms. By \cite{voisin2015rational}, every point of a constant cycle curve on a K3 surface $T$ represents the Beauville--Voisin class $o_T \in \ch_0(T)$ constructed in \cite{BeauvilleVoisin2004}.

A key notion of a constant cycle curve $C \subset T$, introduced by Huybrechts, is its \textit{order} $\ord(C)$: a curve $ C \subset T$ is a constant cycle curve precisely when the diagonal  class  $[\Delta_C] \in \ch_1(T \times C)$ becomes decomposable up to an integer multiple, and $\ord(C)$ is defined as the minimal such positive integer. 
Rational curves on K3 surfaces are constant cycle curves of order one \cite{Huybrechts2014}, and every positive integer can be realized as the order of a smooth elliptic constant cycle curve on some Kummer surface \cite{Huang2025EllipticCCCKummer}. 
While a linear system $|L|$ may contain infinitely  many constant cycle curves \cite[Sec.~6.3]{Huybrechts2014}, one has the following finiteness theorem for fixed order:
\begin{thm}[{\cite[Prop.~5.1]{Huybrechts2014}}] \label{fin_ccc}
	For a complex projective K3 surface $T$, there are at most finitely many constant cycle curves $C \subset T$ of fixed order $N$ in any linear system $|L|$.
\end{thm}

This naturally raises the following question:
\begin{quest}
	Does an analogous finiteness result hold for constant cycle Lagrangian subvarieties on higher-dimensional hyperk\"ahler varieties?
\end{quest}  

In this paper, we answer this question for Fano varieties of lines $F(X)$ of smooth cubic fourfolds $X$. These are hyperkähler fourfolds of $\mathrm{K3}^{[2]}$-type, i.e.\ deformation equivalent to Hilbert schemes of two points of K3 surfaces \cite{BeauvilleDonagi}, and share many features with K3 surfaces. For example, the Chow group $\ch_0(F(X))$  contains a canonical degree-one class $o_F$ \cite{Voisin2008}, which is an analogue of the Beauville--Voisin class on a K3 surface, and every point lying on a constant cycle surface  represents the class $o_F$ \cite{ShenYin2020}.

\subsection{Main results}
We first extend the basic theory of constant cycle curves on K3 surfaces to constant cycle surfaces on $F(X)$; see \Cref{ccs_generaltheory} for details. In particular, for a constant cycle surface $S \subset F(X)$, we define its  \textit{order} $\ord(S)$ to be the minimal positive integer $N$ such that the diagonal class $[\Delta_S] \in \ch^4(F(X) \times S)$ admits a Bloch--Srinivas  decomposition after multiplication by $N$.  Rational surfaces are constant cycle surfaces of order one. 
In general,  determining the order of a non-rational constant cycle surface is difficult.

To formulate an analogous finiteness theorem for constant cycle surfaces on $F(X)$, we restrict to the  family of Fano varieties of lines $F(Y) \subset F(X)$ of hyperplane sections $Y \subset X$. These surfaces  on $F(X)$ play a role similar to  that of curves in a  linear system on a K3 surface: they are Lagrangian \cite{Voisin_1992}, have ample normal bundles for generic $Y \in |\mso_X(1)|$ \cite{Voisin2010coniveau}, and the class  $[F(Y)] \in \ch^2(F(X))$ is independent of $Y$; see also  \cite[Sec.~6.4.3]{Huybrechts_cubicbook} for details.

Our main result is the following.

\begin{thm} \label{main_thm_intro}
	Let $X$ be a smooth cubic fourfold, and let $F(X)$ be its Fano variety of lines. For each integer $N > 0$, there exist at most finitely many hyperplane sections $Y \subset X$ such that the Fano surface $F(Y)$ is a constant cycle surface of order $N$ on $F(X)$. 
\end{thm}

Our proof follows Huybrechts' two-step strategy for  \Cref{fin_ccc}: first,  we derive an Abel--Jacobi criterion that detects constant cycle surfaces and their orders in the intermediate Jacobians;  then, we describe the locus of  $Y \in |\mso_X(1)|$ for which $F(Y)$ is a constant cycle  surface of  fixed order $N$, via the $N$-torsion locus of a normal function. The desired  finiteness then follows from the algebraicity of zero loci of normal functions, established independently by Brosnan--Pearlstein \cite{Brosnan_Pearlstein_2013} and  Schnell \cite{Schnell2012}. 
 
 The main difficulty lies in the first step. A direct adaptation of the strategy for \Cref{fin_ccc} does yield an Abel--Jacobi criterion for constant cycle surfaces $S \subset F(X)$, but the order $\ord(S)$ cannot be read off directly; see \Cref{bridge}. The reason is that the Bloch--Srinivas decomposition of the diagonal $\Delta_S$ is  not a $\mbz$-linear combination of product cycles in general. We resolve this by passing from $F(X)$ to the cubic fourfold $X$ via the Fano correspondence (see \Cref{equi_one_cycle}), and by applying the following result: 

\begin{prop}[\Cref{ccs_iff_ccc_general}] \label{intro_ccsiffccc}
	Let $Y \subset X$ be a hyperplane section that is not a cone, and $l \subset Y$ be a generic line. The Fano variety $F(Y)$ is a constant cycle surface on $F(X)$ if and only if the curve $C_l \subset F(Y)$ of lines intersecting  $l$ is a constant cycle curve on $F(X)$.
\end{prop}

When both $Y$ and $C_l$ are  smooth, the statement follows from the surjectivity of the natural map $\ch_0(C_l) \ra \ch_1(Y)$, as a consequence of  the classical isomorphisms $\mathrm{Alb}(F(Y)) \simeq J(Y)$ \cite{ClemensGriffiths_cubic3fold} and $\ch_1(Y)_{\hom} \simeq J(Y)$ \cite{Bloch-Srinivas83, Murre1985Applications}; the general case follows by specialization  (see \Cref{criterion_ccc_iff_ccs}).

\subsection{Further remarks}
A few  examples of constant cycle surfaces on $F(X)$ are known. For a general cubic fourfold $X$, these include the fixed locus of the Voisin map $\phi \colon F(X) \dashrightarrow F(X)$ constructed in \cite{Voisin02-Kcorrespondence} (see \cite{ShenYin2020}), and the Fano surface $F(Y)$ of a five-nodal hyperplane section $Y \subset X$, which is rational \cite{Voisin2008}. More generally, the Fano surface $F(Y)$ of a sufficiently singular hyperplane section $Y \subset X$ is expected to be rational, hence a constant cycle surface.

Further  candidates come from fixed loci of non-symplectic automorphisms of $F(X)$. For involutions this has been verified in \cite{li2023bloch}. For odd prime order, the generalized Bloch conjecture  (see \cite[Conj.~11.22]{Voisin_HodgeII}) predicts that the fixed locus is a constant cycle subvariety; this has been  confirmed for polarized automorphisms of $F(X)$ whose fixed loci are two-dimensional. Indeed, by \cite{GAL2011automorphisms}, the only possible order is three, and  only two families of cubic fourfolds, $\mcf_3^1$ and $\mcf_3^2$, are involved: in the first case the fixed locus is the Fano variety of a cubic threefold \cite[Ex.~6.4]{BCS16classification}, known to be a constant cycle surface by \cite{Laterveer2018algebraicII}; in the second case the fixed locus consists of rational surfaces \cite[Ex.~6.5]{BCS16classification}.

We are not aware of any  smooth cubic fourfold $X$ such that $F(X)$ contains countably infinitely many constant cycle (Fano) surfaces, and it would be interesting to construct such an example.

\hfill 

\noindent \textbf{Structure of this paper.}  
\Cref{ccs_generaltheory} studies the general theory of constant cycle surfaces on $F(X)$. 
\Cref{ccfs} focuses on Fano surfaces $F(Y)$ of hyperplane sections $Y \subset X$. \Cref{geo_prop}  recalls the geometry of Fano surfaces $F(Y)$ and the curves $C_l \subset F(Y)$.  \Cref{sec_ccs_iff_ccc} proves   \Cref{intro_ccsiffccc}, the characterization of constant cycle Fano surfaces $F(Y)$ in terms of the curves $C_l$; an alternative criterion that applies to all Lagrangian surfaces on $F(X)$ is given in \Cref{appendix}.  \Cref{AJ_Cl} derives the Abel--Jacobi criterion  for constant cycle Fano surfaces $F(Y)$. 
\Cref{sec_finiteness} is devoted to the proof of the finiteness of constant cycle Fano surfaces. \Cref{toolbox} recalls and adapts the theory of normal functions to our setting, and \Cref{proof_finiteness} proves  \Cref{main_thm_intro}. 

\hfill

\noindent \textbf{Acknowledgement.}
The research in this paper was motivated by a question of Emre Sertöz. I would like to thank my advisor Daniel Huybrechts for suggesting this interesting problem, for numerous inspiring discussions, and for valuable feedback that greatly improved the manuscript. I am also grateful to Frank Gounelas for stimulating conversations and helpful comments.

\hfill

\noindent \textbf{Convention.\ }
We work over the field $\mbc$. Subvarieties are assumed to be locally closed. For a scheme $Z$ of dimension $m$, we denote by  $\ch^k(Z) = \ch_{m - k}(Z)$ (resp.\ $\ch^k(Z)_{\mbq} = \ch_{m - k}(Z)_{\mbq}$) the Chow group of cycles of codimension $k$ with integral (resp.\ rational) coefficients. For a cycle $\Gamma \subset X \times Y$ on the product of two varieties,  write $\Gamma^t \subset Y \times X$ for its transpose. When $X$ and $Y$ are smooth and proper, we denote by $\Gamma_{\ast} \colon \ch_{\ast}(X) \to \ch_{\ast}(Y)$ and $\Gamma^{\ast} \colon \ch_{\ast}(Y) \to \ch_{\ast}(X)$ the induced maps on Chow rings.

\section{Constant cycle surfaces} \label{ccs_generaltheory}
Let $X \subset \mbp^5$ be a smooth cubic fourfold, and denote by $F(X)$ its Fano variety of lines. Constant cycle subvarieties contained in the hyperk\"ahler fourfold $F(X)$ have dimension at most two by \cite[Cor.~2.1]{Voisin2016}. In this section, we study the general theory of constant cycle surfaces on $F(X)$. 
We begin by recalling the following pointwise definition.
\begin{defn}
	 \label{def_ccs_pointwise}
	A surface $S \subset F(X)$ is  called a \textit{constant cycle surface} if for any two points $x, y \in S$, we have $[x] = [y]$ in $\ch_0(F(X))$.  
\end{defn}

As observed in \cite{ShenYin2020}, any point on a constant cycle surface $S \subset F(X)$ represents the canonical class $o_F \in \ch_0(F(X))$ constructed by Voisin in \cite[Lem.~3.2]{Voisin2008}. The class $o_F$ has the distinguished property that every zero-cycle class arising as an intersection product of divisors and Chern classes on $F(X)$ is a multiple of $o_F$ \cite[Thm.~1.4]{Voisin2008}. 

\begin{rmk} \label{ccs_is_Lag}
	An integral constant cycle surface $S$ on the Fano variety $F(X)$ is Lagrangian \cite[Cor.~1.2]{Voisin2016}; that is, the restriction of a non-degenerate holomorphic $2$-form on $F(X)$ to the smooth locus of $S$ is trivial.
\end{rmk}

The property of $S \subset F(X)$ being a constant cycle surface can also be detected by passing to the cubic fourfold $X$ via the incidence correspondence $\lx \subset  F(X) \times X$.

For a line $l \subset X$, we use  $[l]$ to denote both the associated one-cycle on $X$ and the corresponding zero-cycle on the Fano variety $F(X)$. By definition, the induced map $[\lx]_{\ast} \colon \ch_0(F(X)) \ra \ch_1(X)$ sends the class $[l]$ to the one-cycle $[l] \in \ch_1(X)$.

The following criterion is a direct corollary of  \cite[Thm.~3.4]{ShenYinZhao2020}.
\begin{lem} \label{equi_one_cycle}
	Let $l_1,l_2$ be two lines contained in a smooth cubic fourfold $X$. Then $[l_1] = [l_2]$ in $\ch_0(F(X))$ if and only if $[l_1] = [l_2]$ in $\ch_1(X)$.
	
	In particular, a surface $S \subset F(X)$ is a constant cycle surface if and only if all lines parametrized by $S$ are rationally equivalent on $X$.  \qed
\end{lem}

\begin{rmk}\label{dist_line_class}
Every line $l \subset X$ parametrized by a constant cycle surface $S \subset F(X)$ represents the distinguished class $[\lx]_{\ast}(o_F) \in \ch_1(X)$. By \cite[Lem.~A.3(v)]{ShenVial2013FTofHK}, one has $3 [\lx]_{\ast}(o_F) = \hx^3$, where $\hx \in \ch^1(X)$ denotes the  hyperplane class. Since $\ch_1(X)$ is torsion-free \cite[Lem.~1.2]{ShenYin2020}, the class $\hx^3$ is uniquely divisible by three, and thus $[\lx]_{\ast}(o_F) = \frac{1}{3}\hx^3$ in $\ch_1(X)$. 
\end{rmk}

We now turn to global characterizations of constant cycle surfaces on $F(X)$. Without loss of generality, we restrict to integral surfaces.

For an integral surface $S \subset F(X)$ with function field $k(S)$, the generic point $\eta_S$ defines a $k(S)$-rational point of the base change $F(X) \times k(S)$. We associate to $S$ the degree-zero class
\begin{equation*}\label{def_kappaS}
	\kappa_S \coloneqq [\eta_S] - o_F \times k(S) \in \ch^4(F(X) \times k(S)).
\end{equation*}
Alternatively, one can describe $\kappa_S$ as follows. Choose a resolution of singularities $\ts \ra S$, and consider the composition $f \colon \ts \ra S \subset F(X)$. Let $\Gamma_f \subset \ts \times F(X)$ denote the graph of $f$, and  $\Gamma_f^t \subset F(X) \times \ts$ its transpose.
Then under the base change 
\begin{equation*} \label{alt_def_kappaS}
	r \colon \ch^4(F(X) \times \ts) \ra \ch^4(F(X) \times k(S))
\end{equation*}
we have
\begin{equation} 
	 r([\Gamma_f^t] - o_F \times [\ts]) =  \kappa_S \in \ch^4(F(X) \times k(S)).
\end{equation}

The following result generalizes  \cite[Lem.~3.4, Prop.~3.7]{Huybrechts2014} to our setting. Note that it also extends to constant cycle Lagrangian subvarieties on hyperkähler varieties, as the main input is the torsion-freeness of $\ch_0$,  a property shared by all  hyperk\"ahler varieties  \cite{Rojtman1980torsion}. 
\begin{lem} \label{equi_def_ccs}
	Let $S \subset F(X)$ be an integral surface. The following statements are equivalent:
	\begin{enumerate}[label=(\roman*), font=\normalfont]
		\item \label{ccs_def} The surface $S$ is a constant cycle surface on $F(X)$.
		\item \label{BS_decomp} There exists a nonzero integer $N$ such that 
		\begin{equation*}
			N[\Gamma_f^t] = N \cdot o_F \times [\ts] + Z \quad  \text{in } \ch^4(F(X) \times \ts),
		\end{equation*}
	 where $Z$ is a cycle supported on $F(X) \times D$ for some proper closed subset $D \subset \ts$.
		\item  \label{eq_def_kappas} The class $\kappa_S \in \ch^4(F(X) \times k(S))$  is torsion. 
		\item \label{alg_closure}The pullback of $\kappa_S$ to $\ch^4(F(X) \times \overline{k(S)})$ is zero.
	\end{enumerate}
\end{lem}
\begin{proof}
		We first prove the equivalence of \ref{ccs_def} and \ref{BS_decomp}. By \Cref{def_ccs_pointwise}, a surface $S \subset F(X)$ is a constant cycle surface if and only if the pushforward $[\Gamma_f]_{\ast} \colon \ch_0(\ts) \to \ch_0(F(X))$ sends every degree-one cycle to the same class. The implication \ref{ccs_def} $\Rightarrow$ \ref{BS_decomp} then follows from the Bloch--Srinivas principle  \cite{Bloch-Srinivas83}. Conversely, assume \ref{BS_decomp}. For any degree-one class $z \in \ch_0(\ts)$, we have $N[\Gamma_f^t]_{\ast}(z) = N o_F$ in $\ch_0(F(X))$; the torsion-freeness of $\ch_0(F(X))$  \cite{Rojtman1980torsion} then implies that the image  $[\Gamma_f^t]_{\ast}(z) = o_F$ is independent of $z$, and hence \ref{ccs_def} holds.
		
		Next, we show the equivalence \ref{BS_decomp} $\Leftrightarrow$ \ref{eq_def_kappas}. The implication  \ref{BS_decomp} $\Rightarrow$  \ref{eq_def_kappas} follows by applying the base change $r \colon \ch^4(F(X) \times \ts) \ra \ch^4(F(X) \times k(S))$. For the converse, note that $r$ is the direct limit of restrictions $r_U \colon \ch^4(F(X) \times \ts) \ra \ch^4(F(X) \times U)$ for $U$ ranging over nonempty open subsets of $\ts$; see \cite[Lem.~1A.1]{Bloch_2010}. 
		 If  $\kappa_S = r([\Gamma_f^t] - o_F \times \ts)$  is $N$-torsion, then  $N([\Gamma_f^t] - o_F \times \ts)$ is supported on $F(X) \times D$ for some proper closed subset $D \subset \ts$, i.e.\ \ref{BS_decomp} holds.
		
		Finally, since  $\ch^4(F(X) \times \overline{k(S)})$ is torsion-free  \cite{Rojtman1980torsion}, and the pullback $\ch^4(F(X) \times k(S)) \to \ch^4(F(X) \times \overline{k(S)})$ has torsion kernel  \cite[Lem.~1A.3]{Bloch_2010}, the class $\kappa_S$ is torsion if and only if its base change to the algebraic closure $\overline{k(S)}$ vanishes. Thus \ref{eq_def_kappas} and \ref{alg_closure} are equivalent. This completes the proof.
\end{proof}

\Cref{equi_def_ccs} shows that the property of $S \subset F(X)$ being a constant cycle surface is completely encoded  by its generic point $\eta_S$. In particular, condition \ref{BS_decomp} in \Cref{equi_def_ccs} is independent of the choice of the resolution $\widetilde{S} \to S$.

Analogous to the case of constant cycle curves on K3 surfaces, one defines the order of a constant cycle surface on $F(X)$ as follows.
\begin{defn} \label{def_order_ccs}
	The \textit{order} of an integral constant cycle surface $S \subset F(X)$, denoted by $\ord(S)$, is defined as the order of the torsion class $\kappa_S \in \ch^4(F(X) \times k(S))$. For an arbitrary constant cycle surface $S \subset F(X)$, the order $\ord(S)$ is defined as the least common multiple of the orders of the irreducible components of the underlying reduced surface $S_{\mathrm{red}}$.
\end{defn}

\begin{rmk} \label{ord_alt_def}
	By the proof of the equivalence \ref{BS_decomp} $\Leftrightarrow$ \ref{eq_def_kappas} in \Cref{equi_def_ccs}, the order of an integral constant cycle surface $S \subset F(X)$ can equivalently be  defined as the minimal positive integer $N$ such that $N[\Gamma_f^t]$ admits a decomposition as in \Cref{equi_def_ccs}\ref{BS_decomp}; this definition is again independent of the choice of resolution $\ts \to S$.
\end{rmk}

Similar to \cite[Lem.~6.1]{Huybrechts2014}, we have the following example.
\begin{ex}  \label{rational_ord1}
	A rational surface $S \subset F(X)$ is a constant cycle surface of order one. Indeed, since $\ch^4(F(X) \times \mbp^2) \simeq \oplus_{i=0}^2 \ch^{4-i}(F(X)) \otimes \ch^{i}(\mbp^2)$, we obtain $$\ch^4(F(X) \times k(S)) \simeq \ch^4(F(X) \times k(\mbp^2)) \simeq \varinjlim_{U \subset \mbp^2} \ch^4(F(X) \times U) \simeq \ch^4(F(X)).$$  Hence $\ch^4(F(X) \times k(S))$ is torsion-free. In particular, the torsion class $\kappa_S = 0$ and $\ord(S) = 1$.
\end{ex}

In general, the order of a constant cycle surface on $F(X)$ is difficult to determine.

We conclude this section with a few observations on constant cycle curves on $F(X)$. 

Let $S \subset F(X)$ be a constant cycle surface. An integral curve $C \subset S$ is a constant cycle curve on $F(X)$. Viewing the generic point $\eta_C \in C$ as a $k(C)$-rational point on the base change $C \times k(C)$, one associates to $C$ the class $$\kappa_C \coloneqq  [\eta_C] - o_F \times k(C) \in \ch^4(F(X) \times k(C)).$$ The same argument as in \Cref{equi_def_ccs} shows that $\kappa_C$ is torsion. Following \cite[Def.~3.5]{Huybrechts2014}, we define  the \textit{order} $\textup{ord}(C)$ of the constant cycle curve $C$ as the order of the torsion class $\kappa_C$. 

The following lemma is immediate.
\begin{lem} \label{order_ccc}
	Let $S \subset F(X)$ be an integral constant cycle surface that is smooth at the generic point of an integral curve $C \subset S$. Then $C$ is a constant cycle curve on $F(X)$ with $\ord(C) \mid \ord(S)$.
\end{lem}
\begin{proof}
	The specialization map  $\textup{sp} \colon \ch^4(F(X) \times k(S)) \ra \ch^4(F(X) \times k(C))$ (see \cite[Sec.~20.3]{fulton1998intersection}) sends $\kappa_S$ to $\kappa_C$. If $S$ is a constant cycle surface of order $N$, i.e.\  the associated  torsion class $\kappa_S$ has order $N$, then  $\kappa_C$ is also  $N$-torsion.
\end{proof}
\begin{rmk} 
	If a constant cycle curve $C \subset F(X)$ is not contained in any constant cycle surface, the image of 
	$\ch_0(C) \to \ch_0(F(X))$ need not be generated by the canonical class $o_F$. 
	A general rational curve in the ruling of the uniruled divisor $D \subset F(X)$ constructed in \cite[Prop.~4.4(b)]{Voisin2016} provides such an example, since $F(X)$ contains no constant cycle divisors. 
	Nevertheless, the order of such a constant cycle curve $C$ remains well-defined, by replacing $o_F$ with the degree-one generator of the image of $\ch_0(C) \to \ch_0(F(X))$ in the definition of $\kappa_C$.
\end{rmk}

\section{Constant cycle Fano surfaces} \label{ccfs}
Let $X$ be a smooth cubic fourfold with $F(X)$ its Fano variety of lines. In this section, we focus on constant cycle surfaces on $F(X)$ arising as Fano varieties $F(Y) \subset F(X)$ of hyperplane sections $Y \subset X$. We prove a  criterion for $F(Y)$ to be a constant cycle surface, in terms of the curve  $C_l \subset F(Y)$ of lines intersecting a generic line $l \subset Y$ (see \Cref{ccs_iff_ccc_general}), and derive an Abel--Jacobi criterion for constant cycle Fano surfaces (see \Cref{cri_AJ_Cl}).
\subsection{Geometry of Fano surfaces} \label{geo_prop}
We begin by recalling the geometry of Fano surfaces of hyperplane sections of a smooth cubic fourfold $X$.

Let $Y \subset X$ be a hyperplane section. Then $Y$ is a cubic threefold with at most isolated singularities. By \cite[Sec.~1]{AltmanKleiman_Fanoscheme}, its Fano variety of lines $F(Y)$ is a connected locally complete intersection of pure dimension two, and the singular locus $\mathrm{Sing}(F(Y))$ consists precisely of the points that correspond to lines through a singular point of $Y$. The following three cases occur:

\begin{enumerate}[label=(\roman*), font=\normalfont]
	\item 
	If  $Y$ is smooth, then $F(Y)$ is a smooth irreducible surface.	
	\item If $Y$ is singular with no triple points, then through each singular point of $Y$ there passes a one-dimensional family of lines contained in $Y$, so $\dim \sing(F(Y)) = 1$. The surface $F(Y)$ is geometrically reduced, though possibly  reducible.
	\item \label{cone_case}If $Y$ has a triple point $O$, then $Y$ is a cone over a smooth cubic surface $S$ with vertex $O$ and contains $27$ planes.  In this case, a line $l \subset X$ is contained in $Y$ if and only if either $O \in l$ or $l$ lies on one of the $27$ planes. The Fano surface $F(Y)$ is nonreduced, and the underlying reduced surface $F(Y)_{\rm red}$ has $28$ irreducible rational components. 
\end{enumerate}

\begin{rmk} \label{exclude_cone}
	Case \ref{cone_case} occurs precisely when $X$ has an \textit{Eckardt point}\footnote{Also called a \textit{star point} in \cite[Def.~2.2]{CollsCoppensstarpoints}.}, 
	that is, a point through which passes a two-dimensional family of lines on $X$. Such cubic fourfolds admit  non-symplectic involutions and form a $14$-dimensional family in the  moduli space  \cite[Thm.~3.8]{GAL2011automorphisms}. By \cite[Lem.~1.6]{LPZ18}, in suitable coordinates, the cubic fourfold $X$ is given by the equation $x_5^2x_0 + G(x_0,\ldots,x_4)=0$
	with Eckardt point $O=[0:\cdots:0:1]$, and the hyperplane section $Y = X \cap V(x_0)$ is the unique cone with vertex $O$. 
	Since  a smooth cubic fourfold  has at most  finitely many Eckardt points  \cite[Cor.~2.2]{CoskunStarr09}, only finitely many hyperplane sections $Y \subset X$ are cones, and  thus their Fano surfaces contribute only finitely many constant cycle surfaces of order one on $F(X)$ (see \Cref{rational_ord1}).
\end{rmk}

For a line $l \subset Y$, define $C_l \subset F(Y)$ as the reduced scheme supported on the closure of the set
\begin{equation} \label{def_Cl}
	\{[l^{\prime}] \in F(Y) \mid l^{\prime} \subset Y \ \textup{with} \ l^{\prime} \cap l \neq \varnothing\}.
\end{equation} 
When $Y$ is smooth, the scheme $C_l$  is a curve, and $[l] \in C_l$ if and only if $l$ is a line of \textit{second type}, i.e.\ there exists a unique plane in $\mbp^4$ tangent to $Y$ along $l$; see \cite[Lem.~10.7]{ClemensGriffiths_cubic3fold}. Further, if $l \subset Y$ is \textit{non-special}\footnote{Such a line is also referred to as a \textit{good line} in \cite[Def.~2.2]{LazaSaccaVoisin}.} in the sense  of \cite[Def.~3.4]{CM-Laza09}, namely $l$ is neither of second type nor the residual line of a line of second type, then the curve $C_l$ is smooth and irreducible  \cite[p.~326]{ClemensGriffiths_cubic3fold}.

In general,  the subscheme $C_l \subset F(Y)$ can be: 
\begin{enumerate}[label=(\roman*)]
	\item singular, e.g.\ when $Y$ is smooth and $l \subset Y$ is of second type;
	\item reducible, e.g.\ when $Y$ is singular and $l$ passes through a singular point of $Y$;
	\item not equidimensional, e.g.\ when $Y$ contains a plane $\Pi$ and $l \subset \Pi$; in this case, the scheme $C_l$ contains a two-dimensional rational component. \label{nonequidim}
\end{enumerate}

The most degenerate case occurs when $Y \subset X$ is a cone: for every line $l \subset Y$, the scheme $C_l$ contains at least one  two-dimensional component. The following lemma  shows that this happens precisely when $Y$ is a cone. 
\begin{lem} \label{existence_1dim}
	Let  $Y \subset X$ be a hyperplane section with no triple points. Then there  exists an irreducible component $S \subset F(Y)$ in which the locus $\{[l] \in S \mid \dim(C_l) = 1\}$ 
	is Zariski open and dense. We call a line $l \subset Y$ generic if $\dim(C_l) = 1$.
\end{lem}
\begin{proof}
	This result is likely classical; however, as we were unable to locate a reference, we include a proof for completeness. It suffices to show that there exists a line $l \subset Y$ such that $\dim(C_l) = 1$. Suppose for contradiction that $\dim(C_l) = 2$ for every line $l \subset Y$. For each irreducible component $S_i \subset F(Y)$, there exists a (possibly identical) irreducible component $S_j \subset F(Y)$ such that for a generic point $[l] \in S_i$, the scheme $C_l$ contains $S_j$. Then, for any two generic points $[l_1], [l_2] \in S_i$, the lines $l \subset Y$ meeting both $l_1$ and $l_2$ form a two-dimensional family. It follows that  $l_1 \cap l_2 \neq \varnothing$, and the lines parametrized by $S_i$ sweep out either a plane or a cone contained in $Y$. Since the former cannot occur for all irreducible components of $F(Y)$, the cubic threefold $Y$ must be a cone, a contradiction.
\end{proof}

In what follows, to simplify the discussion, we consider only hyperplane sections $Y \subset X$ that are not cones and  lines $l \subset Y$ that are generic in the above sense. 

\subsection{A criterion via the curve $C_l$}\label{sec_ccs_iff_ccc} 
 The aim of this section is to prove the following  characterization of constant cycle Fano surfaces on $F(X)$.
\begin{prop}[\Cref{intro_ccsiffccc}] \label{ccs_iff_ccc_general}
	Let $Y \subset X$ be a  hyperplane section with no triple points, and let $l \subset Y$ be a generic line in the sense of \Cref{existence_1dim}. The following statements are equivalent:
	\begin{enumerate}[label=(\roman*), font=\normalfont]
		\item \label{ccs} The Fano surface $F(Y)$ is a constant cycle surface on $F(X)$.
		\item \label{ccc}The curve $C_l \subset F(Y)$  is a constant cycle curve on $F(X)$.
	\end{enumerate} 
\end{prop} 

The following corollary is immediate. Note that compared with condition~\ref{BS_decomp}
of \Cref{equi_def_ccs}, the Chow-theoretic decomposition in
condition~\ref{ccc_decomp} below involves only product cycles.

\begin{cor} \label{ccs_decomp_curve}
	In the setup of \Cref{ccs_iff_ccc_general}, let $\tcl \ra C_l$ be the normalization, and let $\tpl  \subset \tcl \times X$ denote the incidence correspondence of lines on $X$ parametrized by $\tcl$.
	The following statements are equivalent:
	\begin{enumerate}[label=(\roman*), font=\normalfont]
		\item \label{ccs_2}The Fano surface $F(Y)$ is a constant cycle surface on $F(X)$.
		\item \label{ccc_decomp} There exists an integer $N > 0$ such that 	\begin{equation*}
			N[\tpl^t] = N \cdot \frac{1}{3}\hx^3 \times [\tcl] + Z  \ \ \textup{in} \ \ch^3(X \times \tcl),
		\end{equation*}
		where $Z$ lies in the image of $\ch^2(X) \otimes \ch^1(\tcl) \ra \ch^3(X \times \tcl)$. 
	\end{enumerate}
	In this case, one can take the coefficient $N$ to be the order of the constant cycle surface $F(Y)$ (see \Cref{def_order_ccs}).
\end{cor}
\begin{proof}
	Assume  \ref{ccc_decomp} holds. The pushforward $N[\tpl]_{\ast} \colon \ch_0(\tcl) \ra \ch_1(X)$ sends every degree-one class to $N \cdot \frac{1}{3} \hx^3$. By the torsion-freeness of $\ch_1(X)$  \cite[Lem.~1.2]{ShenYin2020} and   \Cref{equi_one_cycle}, the curve  $C_l$ is a constant cycle curve on $F(X)$. Hence $F(Y) \subset F(X)$ is a constant cycle surface by \Cref{ccs_iff_ccc_general}.
	
	Conversely, suppose $F(Y) \subset F(X)$ is a constant cycle surface of order $N$.
	We may assume $C_l$ is irreducible, as the general case follows by applying the same argument to each irreducible component. By \Cref{order_ccc},  the class $\kappa_{C_l} = [\eta_{C_l}] - o_F \times k(C_l) \in \ch^4(F(X) \times k(C_l))$ is $N$-torsion. Let $\Gamma_l$ denote the graph of $f_l \colon \tcl \ra F(X)$. Using  $\ch^4(F(X) \times k(C_l)) \simeq \varinjlim_{U \subset C_l} \ch^4(F(X) \times U)$  \cite[Lem.~1A.1]{Bloch_2010}, and arguing as in the proof of  \ref{eq_def_kappas} $\Rightarrow$ \ref{BS_decomp} in  \Cref{equi_def_ccs}, we obtain
	\begin{equation} \label{dec_Gamma_l}
		N[\Gamma_l^t] = N \cdot o_F \times [\tcl] + Z^{\prime} \quad \text{in } \ch^4(F(X) \times \tcl),
	\end{equation}	
	where $Z^{\prime}$ lies in the image of the natural map	$\ch^3(F(X)) \otimes \ch^1(\tcl) \ra \ch^4(F(X) \times \tcl)$. 
	Composing with the dual correspondence $[\lx^t] \in  \ch^3(X \times F(X))$ gives
	\begin{equation*}
		N[\tpl^t] = N \cdot \frac{1}{3}\hx^3 \times [\tcl] + Z \quad \textup{in } \ch^3(X \times \tcl),
	\end{equation*}
	where $Z$ lies in the image of $\ch^2(X) \otimes \ch^1(\tcl) \ra \ch^3(X \times \tcl)$. This completes the proof.
\end{proof}

Before we proceed to the proof of \Cref{ccs_iff_ccc_general}, let us fix some notations. Let $i_l \colon C_l \hookrightarrow F(Y)$, $i_F \colon F(Y) \hookrightarrow F(X)$, and $i \colon Y \hookrightarrow X$ denote the natural inclusions. Write $\ly \subset F(Y) \times Y$ and $P_l \subset C_l \times X$ for the incidence correspondences. We have the following commutative diagram:
 \begin{equation} \label{diag_Fano_corr}
 	\begin{tikzcd}
 		&\ch_0(C_l) \arrow[r,"(i_{l})_{\ast}"]& \ch_0(F(Y)) \arrow[d,"(i_{F})_{\ast}"] \arrow[r,"{[\ly]_{\ast}}"]
 		&  	\ch_1(Y) \arrow[d,"i_{\ast}"] \\
 		& &\ch_0(F(X))  \arrow[r, "{[\lx]_{\ast}}"]
 		& \ch_1(X) .
 	\end{tikzcd}
 \end{equation}
The composition  $i_{\ast} \circ [\ly]_{\ast} \circ (i_l)_{\ast}$ is simply the induced map $[P_l]_{\ast} \colon \ch_0(C_l) \ra \ch_1(X)$.

\begin{proof}[Proof of \Cref{ccs_iff_ccc_general}]
	Clearly, \ref{ccs} implies \ref{ccc}. Conversely, suppose $C_l$ is a constant cycle curve on $F(X)$.  Then the map $[P_l]_{\ast} = i_{\ast} \circ [\ly]_{\ast} \circ (i_l)_{\ast}$ is constant on degree-one  classes. By \Cref{criterion_ccc_iff_ccs} below, the same holds for  $[\lx]_{\ast} \circ (i_F)_{\ast}$. It then follows from \Cref{equi_one_cycle} that $F(Y)$ is a constant cycle surface on $F(X)$. Hence,  \ref{ccs} and \ref{ccc} are equivalent.
\end{proof}

The proof of \Cref{ccs_iff_ccc_general} relies on the following lemma.
\begin{lem} \label{criterion_ccc_iff_ccs}
	In the setup of \Cref{ccs_iff_ccc_general}, the two maps $[P_l]_{\ast}$ and $[\lx]_{\ast} \circ (i_F)_{\ast}$ in \eqref{diag_Fano_corr} have the same image in $\ch_1(X)$. 
\end{lem}

\begin{proof}[Proof of \Cref{criterion_ccc_iff_ccs}]
	 We call $(Y, l)$ a \textit{good pair} if $Y \subset X$ is a smooth hyperplane section, and $l \subset Y$ is a line such that $C_l$ is a smooth curve. Let us first prove the claimed property for a good pair $(Y, l)$. We use the following  commutative diagram of Abel--Jacobi maps:
		\begin{equation*}
			\adjustbox{scale=1,center}{
				\begin{tikzcd}
					& \ch_0(C_l)_{\lhom} \arrow[r, "(i_l)_{\ast}"] \arrow[d, "\simeq"]
					& \ch_0(F(Y))_{\lhom} \arrow[d] \arrow[r, "{[\mathbb{L}_Y]_{\ast}}"] 
					& \ch_1(Y)_{\lhom} \arrow[d, "\Phi_Y"] \\
					& \textup{Alb}(C_l) \arrow[r]& \textup{Alb}(F(Y)) \arrow[r,"\sim"] & J(Y) .
				\end{tikzcd}
			}
		\end{equation*}
		The morphism $\textup{Alb}(C_l) \ra \textup{Alb}(F(Y))$  of Albanese varieties is surjective, and $\textup{Alb}(F(Y)) \simeq J(Y)$ \cite[Thm.~11.19]{ClemensGriffiths_cubic3fold}. Further, the Abel--Jacobi map $\Phi_Y$ is an isomorphism by \cite[Thm.~1]{Bloch-Srinivas83} together with  \cite[Thm.~A and ~C]{Murre1985Applications}. Hence  $[\ly]_{\ast} \circ (i_l)_{\ast}$ is also surjective. In particular,   $[P_l]_{\ast}$ and $[\lx]_{\ast} \circ (i_F)_{\ast}$ have the same image. 
		
	Now let us consider an arbitrary pair $(Y, l)$, i.e.\ either $Y$ or $C_l$ is singular. Set $(\mcy_0, l_0) \coloneqq (Y, l)$. We need to show that for any line $l_0^{\prime} \subset \mcy_0$, there exists $\alpha_0 \in \ch_0(C_{l_0})$ such that $[P_l]_{\ast}(\alpha_0) =  [l_0^{\prime}]$ in $\ch_1(X)$. Our strategy is to deform the pair $(\mcy_0, l_0)$ to good pairs $(\mcy_t, l_t)$, for which the desired property has been verified, and then apply a specialization argument.

	\begin{enumerate}[font=\normalfont, leftmargin=2em, topsep=0pt, partopsep=0pt, itemsep=0pt]
		\item Let $\mcy \ra \mbp \coloneqq |\mso_X(1)|$ be the universal family of hyperplane sections of $X$, and $F(\mcy/\mbp) \ra \mbp$
		 the relative Fano variety of lines. Pick any good pair $(\mcy_1, l_1)$. By Bertini's theorem,  there exists a curve $T \subset F(\mcy/\mbp)$ containing the  points $[l_1] \in F(\mcy_1)$ and $[l_0] \in F(\mcy_0)$. Let $\bar{T}$ denote the image of $T$ under $F(\mcy/\mbp) \ra \mbp$, and consider the one-dimensional family $\mcy_{\bar{T}} \ra \bar{T}$ of cubic threefolds. After a finite base change and, if necessary,  shrinking the base, we obtain a family $\mcy_{B} \ra B$ of hyperplane sections of $X$ with the following properties:
			\begin{itemize}
				\item  The base $B$ is a quasi-projective smooth curve that contains $0$.
				\item For $b \in B_0 \coloneqq B \backslash 0$, the fibre $\mcy_b \subset X$ is a smooth cubic threefold. 
				\item  The relative Fano variety $F(\mcy_{B}/B) \ra B$ admits two sections $s, s^{\prime} \colon B \ra F(\mcy_{B}/B)$ with $s(0) = [l_0]$ and $s^{\prime} (0)= [l_0^{\prime}]\in F(\mcy_0)$. For $b \in B$, we denote by $l_b, l_b^{\prime} \subset \mcy_b$ the lines corresponding to the points $s(b), s^{\prime}(b) \in F(\mcy_b)$,  respectively. 
				\item The section $s$ determines a family of curves $\mcc_{B} \subset F(\mcy_B/B) \to B$: for each $b \in B$, the fibre $\mcc_b  \subset F(\mcy_b)$ is the curve $C_{l_b}$ of lines contained in $\mcy_b$ that meet $l_b$ (see \eqref{def_Cl}). The family $\mcc_B \ra B$ has smooth fibres over $B_0$, and hence  $(\mcy_b, l_b)$ is a good pair for every $b \in B_0$. The central fibre $\mcc_0 \coloneqq C_{l_0}$ may be   singular.
			\end{itemize} 
		\item The natural inclusion $\mcc_{B_0} \subset F(\mcy_{B_0}/B_0)$ over $B_0$  induces the following  morphsims of relative Picard varieties, Albanese varieties and intermediate Jacobians:
		\begin{equation*}
			f \colon \pic^0(\mcc_{B_0}/B_0) \simeq \textup{Alb}(\mcc_{B_0}/B_0) \ra  \textup{Alb}(F(\mcy_{B_0}/B_0)/B_0) \simeq J(\mcy_{B_0}/B_0),
		\end{equation*}
		where the latter isomorphism is the relative version of \cite[Thm.~11.19]{ClemensGriffiths_cubic3fold}. For each $b \in B_0$, since the pair $(\mcy_b, l_b)$ is good,  $f$ restricts to a surjective morphism $$f_b \colon \pic^0(\mcc_b) \simeq \ch_0(\mcc_b)_{\lhom} \ra \ch_1(Y)_{\lhom} \simeq J(\mcy_b),$$ so $f$ is itself surjective. Now we choose a line $l_0^{\prime \prime} \subset \mcy_0$ that meets $l_0$. After replacing $B$ by a finite base change, we can assume the family $\mcc_B \ra B$ admits a third section $s^{\prime\prime}$ with  $s^{\prime \prime}(0) = [l_0^{\prime \prime}] \in \mcc_0$. 
		Then  $[\textup{im}(s^{\prime})] - [\textup{im}(s^{\prime \prime})] \in \ch^2(F(\mcy_{B}/B))$ gives a  family of cohomologically trivial cycles $\{[l_b^{\prime}] - [l_b^{\prime \prime}] \in \ch^2(\mcy_b)_{\hom} \mid b \in B\}$, and hence yields a section $s_J \colon B_0 \ra J(\mcy_{B_0}/B_0)$ of the relative intermediate Jacobian. The surjectivity of $f$ allows us to lift $s_J$ to a section $s_{P} \colon B_0 \ra \pic^{0}(\mcc_{B_0}/B_0)$ of the relative Picard variety (again after a finite base change of $B$), which further lifts to a class $\alpha \in \ch^1(\mcc_{B_0})$. 
		\item Let $b \in B_0$ and write $\alpha_{b}  \in \ch^1(\mcc_{b})$ for the restriction of $\alpha$ to the fibre $\mcc_b$. Let $Z^{\prime}, Z^{\prime \prime} \subset X \times B$ be the incidence correspondences determined by the sections $s^{\prime},s^{\prime \prime} \colon B \ra F(\mcy_B/B)$. Set $\bar{\beta} \coloneqq [Z^{\prime}] - [ Z^{\prime \prime}] \in \ch^3(X \times B)$ and $\beta \coloneqq \bar{\beta}|_{X \times B_0} \in \ch^3(X \times B_0)$. Consider the following commutative diagram of Chow groups
		\begin{equation*}
			\begin{tikzcd}
				g \colon \ \ \ch^1(\mcc_{B_0}) \arrow[r] \arrow[d]
				& \ch^2(\mcy_{B_0}) \arrow[r] \arrow[d]
				& \ch^3(X \times B_0) \arrow[d] \\
				g_b \colon \ \ \ch^1(\mcc_{b}) \arrow[r]
				& \ch^2(\mcy_{b}) \arrow[r]
				& \ch^3(X \times b),
			\end{tikzcd}
		\end{equation*}
	where the vertical maps are Gysin maps.  Denote by $(\cdot)_{b}$ the image of a class under the Gysin maps.
	The classes $g(\alpha)$ and $\beta \in  \ch^3(X \times B_0)$ satisfy  $$g(\alpha)_b = g_b(\alpha_b) =  [l_b^{\prime}] - [l_b^{\prime\prime}] = \beta_b \quad \text{in } \ch^3(X \times b)$$ for every $b \in B_0$. It then follows from \cite{Bloch-Srinivas83} that $g(\alpha) = \beta + \theta$ in $\ch^3(X \times B_0)$ for some torsion class $\theta$, after possibly shrinking $B_0$.
	\item Now let us consider  the following  diagram with vertical arrows given by the specialization map to the fibre of $\mcc_B \ra B$ over $0$ (see \cite[Sec.~20.3]{fulton1998intersection}):
	\begin{equation*}
		\begin{tikzcd}
			g \colon \ \ \ch^1(\mcc_{B_0}) \arrow[r] \arrow[d, "\textup{sp}"'] & \ch^4(F(X) \times B_0) \arrow[d, "\textup{sp}"] \arrow[r] & \ch^3(X \times B_0) \arrow[d, "\textup{sp}"] \\
			g_0  \colon  \ \ \ch^1(\mcc_{0}) \arrow[r]  
			& \ch^4(F(X) \times 0) \arrow[r]   
			& \ch^3(X \times0).                                    
		\end{tikzcd}
	\end{equation*}
	Since  specializations commute with proper pushforwards and flat pullbacks, the diagram is commutative.   Hence $$g_0 (\textup{sp}(\alpha)) = \textup{sp}(g(\alpha)) = \textup{sp}(\beta) + \textup{sp}(\theta) \ \in \ch^3(X \times 0) \simeq \ch_1(X).$$
	Because $\ch_1(X)$ is torsion-free \cite[Lem.~1.2]{ShenYin2020} and $\theta$ is torsion, we obtain $\textup{sp}(\theta) = 0$. Furthermore, $$\textup{sp}(\beta) = \bar{\beta}_0 = [Z^{\prime}]_0 - [Z^{\prime \prime}]_0 = [l_0^{\prime}] - [l_0^{\prime\prime}] \ \in \ch_1(X).$$ Thus  $g_0(\textup{sp}(\alpha)) = [l_0^{\prime}] - [l_0^{\prime\prime}] \in \ch_1(X)$. Since  $[l_0^{\prime\prime}] \in \mcc_0$ by construction, the class $$\alpha_0 \coloneqq  \textup{sp}(\alpha) + [l_0^{\prime\prime}]\in \ch_0(\mcc_0)$$ satisfies that $g_0(\alpha_0) = [l_0^{\prime}] \in \ch_1(X)$. This completes the proof.	\end{enumerate} 
	\vspace{-\baselineskip}
\end{proof}

\subsection{Abel--Jacobi maps} \label{AJ_Cl}
Let $X$ be a smooth cubic fourfold and $Y \subset X$ a hyperplane section with no triple points. For a generic line $l \subset Y$ in the sense of \Cref{existence_1dim}, the subscheme $C_l \subset F(Y)$ defined in \eqref{def_Cl} is a curve with normalization $\tcl$. Using the Abel–Jacobi map, we derive from \Cref{ccs_decomp_curve} a necessary condition for $F(Y)$ to be a constant cycle surface on $F(X)$ in the intermediate Jacobian $J^5(X \times \tcl)$.

We fix the setup. Consider the Abel--Jacobi map
\begin{equation} \label{AJ_def}
	\Phi \colon \ch^3(X \times \tcl)_{\lhom} \ra J^5(X \times \tcl) = \frac{\big(F^3(H^4(X, \mbc) \otimes H^1(\tcl, \mbc))\big)^{\ast}}{H^4(X, \mbz) \otimes H^1(\tcl, \mbz)}, \quad Z \mapsto \int_{\Gamma}, \ \partial \Gamma = Z.
\end{equation}
Let $J_{\rm{alg}}^5(X \times \tcl)$ denote the algebraic part of $J^5(X \times \tcl)$, i.e.\ the largest complex subtorus of $J^5(X \times \tcl)$ whose tangent space lies in $H^{2,3}(X \times \tcl)$. It receives the Abel--Jacobi images of cycles that are algebraically equivalent to zero (see e.g.\ \cite[Sec.~12.2]{Voisin_2002HodgeI}),  in particular those in the image of  $\ch^2(X) \otimes \ch^1(\tcl)_{\lhom} \to \ch^3(X \times \tcl)_{\lhom}$. 
Set 
\begin{equation} \label{eq_tr_coh}
	H_{\rm{tr}}^4(X,\mbz) \coloneqq H^4(X,\mbz)/H^{2,2}(X,\mbz), \quad H_{\rm{tr}}^{2,2}(X) \coloneqq H^{2,2}(X)/H^{2,2}(X,\mbz) \otimes \mbc.
\end{equation} The \textit{transcendental intermediate Jacobian} is by definition
\begin{equation*}
	J_{\rm{tr}}^5(X \times \tcl) \coloneqq \frac{J^5(X \times \tcl)}{J_{\rm{alg}}^5(X \times \tcl)} \simeq\frac{(H^{3,1}(X) \otimes H^1(\tcl,\mbz) \oplus H_{\rm{tr}}^{2,2}(X) \otimes H^{1,0}(\tcl))^{\ast}}{H_{\rm{tr}}^4(X,\mbz) \otimes H^1(\tcl, \mbz)}.
\end{equation*}
We have the following commutative diagram with exact rows:
\begin{equation} \label{diag_AJ}
	\begin{tikzcd}[sep=1.8em, font=\small]
		& &\ch^2(X) \otimes \ch^1(\tcl)_{\lhom} \arrow[r] \arrow[d,"\Phi_{\textup{alg}}"] &\ch^3(X \times \tcl)_{\lhom} \arrow[r] \arrow[d,"\Phi"] &\frac{\ch^3(X \times \tcl)_{\lhom}}{\ch^2(X) \otimes \ch^1(\tcl)_{\lhom}} \arrow[r] \arrow[d] &0\\
		&0 \arrow[r] &J_{\textup{alg}}^5(X \times \tcl) \arrow[r] &J^5(X \times \tcl) \arrow[r, "p_{\textup{tr}}"] &J_{\textup{tr}}^5(X \times \tcl) \arrow[r] & 0.
	\end{tikzcd}
\end{equation}
Since  $\ch^2(X)_{\lhom} = 0$  \cite[Thm.~1]{Bloch-Srinivas83} and  the integral Hodge conjecture holds for smooth cubic fourfolds  \cite[Thm.~18]{Voisin2007-HodgeConj}, the cycle class map induces an isomorphism  $\ch^2(X) \simeq H^{2,2}(X,\mbz)$. As $J_{\textup{alg}}^5(X \times \tcl) \simeq H^{2,2}(X,\mathbb{Z}) \otimes J(\tcl)$, the map $\Phi_{\rm{alg}}$ is  an isomorphism. 
We call the composition \begin{equation} \label{tr_AJ_def}
	\Phi_{\rm tr} \coloneqq p_{\rm{tr}} \circ \Phi \colon \ch^3(X \times \tcl)_{\lhom} \ra J_{\rm{tr}}^5(X \times \tcl)
\end{equation} the \textit{transcendental Abel--Jacobi map}\footnote{Our definition of the transcendental Abel–Jacobi map  differs from the standard one, where the source is the Griffiths group $\textup{Griff}(X \times \ts)$, i.e.\ the quotient of $\ch^3(X \times \ts)_{\lhom}$ 
by the subgroup of cycles that are algebraically equivalent to zero.}.

From now on, fix an orthogonal $\mbq$-basis  \begin{equation} \label{def_basis}
	\{[S_1], \ldots,[S_{\rho}] \mid [S_i] \in \ch^2(X) \simeq H^{2,2}(X,\mbz)\}
\end{equation}
of $\ch^2(X)_{\mbq} \simeq H^{2,2}(X,\mbq)$ with $[S_1] = \hx^2$, where $\hx \in \ch^1(X)$ is the hyperplane class.  Then $\oplus_{i=1}^{\rho} \mbz [S_i] \subset \ch^2(X)$ is a sublattice of finite index.  Let $\tpl^t \subset  X \times \tcl$ denote the transpose of the incidence correspondence $\tpl \subset \tcl \times X$ (see \Cref{ccs_decomp_curve}), and set  $\theta_i \coloneqq [\tpl^t]_{\ast}[S_i] \in \ch^1(\tcl)$.  Note that $\theta_1 = f_l^{\ast}(\gx)$ is the pullback of the Plücker polarization $\gx \in \ch^1(F(X))$ via  $f_l \colon \tcl \to C_l \hookrightarrow F(X)$; see e.g.\ \cite[Lem.~2.5.1]{Huybrechts_cubicbook}.

Let $[S_i]^2$ denote the self intersection number of $[S_i] \in \ch^2(X)$, and consider the following  algebraic cycles:
\begin{equation} \label{def_Zl}
	\delta_l \coloneqq \sum_{i=2}^{\rho} \frac{1}{[S_i]^2}[S_i] \times \theta_i, \quad 
	Z_l \coloneqq [\tpl^t] - \frac{1}{3}\hx^3 \times [\tcl]  - \frac{1}{3}\hx^2 \times f_l^{\ast}(\gx) - \delta_l \in \ch^3(X \times \tcl)_{\mbq}.
\end{equation} 
 The integer $\ix \coloneqq \prod_{i=1}^{\rho} [S_i]^2$ is  independent of the surface $S \subset F(X)$. By construction, both $\ix \delta_l$ and $\ix Z_l$ are  integral cycles, i.e.\ they belong to $\ch^3(X \times \tcl)$.

\begin{lem}
	\label{decomp_Pl}
	Both $\ix  Z_l$ and $\ix  \delta_l \in \ch^3(X \times \tcl)$  are  cohomologically trivial. 
\end{lem}
\begin{proof}
	It suffices to treat the case where $Y$ is smooth and $l \subset Y$ is non-special, i.e.\ $C_l \subset F(Y)$ is a smooth and connected curve. The statement for a general pair $(Y^{\prime}, l^{\prime})$ then follows by deforming to such a pair and applying a specialization argument.
	
	We first show that the class 
	\begin{equation*}
		\ix(Z_l + \delta_l)= \ix \big([P_l^{t}] - \frac{1}{3}\hx^3 \times [C_l]  - \frac{1}{3}\hx^2 \times f_l^{\ast}(\gx) \big) \in \ch^3(X \times C_l)
	\end{equation*} is cohomologically trivial. Since $X$ has no odd cohomology, it suffices to verify that the induced map $$\ix ([Z_l] + [\delta_l])^{\ast} \colon H^{2k}(C_l,\mbz) \ra H^{2k+2}(X,\mbz)$$  vanishes for $k = 0,1$. We compute the following map (see \eqref{diag_Fano_corr} for the notations): 
		\begin{equation} \label{factorization_Pl}
		[P_l]_{\ast} \colon \ch_i(C_l) \xrightarrow{(i_{l})_{\ast}} \ch_i(F(Y)) \xrightarrow{[\ly]_{\ast}} \ch_{i+1}(Y) \xrightarrow{i_{\ast}} \ch_{i+1}(X), \ \ i = 0, 1.
	\end{equation}
	
	For $i = 0$, the map $[P_l]_{\ast}$ sends a point  of $C_l$ to the class of the corresponding line. Since  every line on $X$ has cohomology class $\frac{1}{3} \hx^3 \in H^6(X,\mbz)$, the map $\ix ([Z_l] + [\delta_l])^{\ast}$ vanishes  on $H^2(C_l, \mbz)$. 
	
	For $i = 1$, as $\ch_2(Y)$ is generated by the hyperplane class $\hy$, there exists an integer $n$  such that  $[\ly]_{\ast}[C_l] = n \cdot \hy \in \ch_2(Y)$. Let $[\ly^t]_{\ast} \colon \ch_1(Y) \ra \ch_1(F(Y))$ be the map induced by the dual Fano correspondence. By \cite[(1.1)]{ClemensGriffiths_cubic3fold}, the  intersection products \begin{equation*}
		([\ly]_{\ast}[C_l]) \cdot [l] = n \hy \cdot [l] \in \ch_0(Y), \qquad [C_l] \cdot ([\ly^t]_{\ast}[l]) = [C_l] \cdot [C_l] \in \ch_0(F(Y)) 
	\end{equation*} have the same degree. Since $(\hy \cdot l) = 1$ and $(C_l \cdot C_l) = 5$, we obtain  $n = 5$. Thus, by  \eqref{factorization_Pl} and  the projection formula, \begin{equation*}
	[P_l]_{\ast}[C_l] = i_{\ast}(5\hy) = 5\hx^2 \quad \textup{in } \ch_2(X).
	\end{equation*} Note that $f_l^{\ast}(\gx) \in \ch_0(C_l)$  is  the restriction of the Pl{\"u}cker polarization $\gy$ on $F(Y)$. Since  $\gy = 3[C_l]$ in $H^{2}(F(Y),\mbz)$  \cite[Sec.~10]{ClemensGriffiths_cubic3fold}, we have $\deg f_l^{\ast}(\gx) = 3(C_l \cdot C_l) = 15$.  Hence
	\begin{equation} \label{loc_corr_H0}
		\ix \big( [Z_l] + [\delta_l])^{\ast}[C_l] = \ix ([P_l]_{\ast}[C_l] - \frac{1}{3}\deg f_l^{\ast}(\gx) \cdot \hx^2\big) = 0 \quad \textup{in }  H^{4}(X,\mbz),
	\end{equation} 
	 i.e.\ $\ix ([Z_l] + [\delta_l])^{\ast}$  also vanishes on $H^0(C_l, \mbz) \simeq \mbz [C_l]$.

	We have shown so far that $\ix ([Z_l] + [\delta_l]) = 0$ in $H^6(X \times \tcl, \mbz)$. It remains to prove that   $\ix \delta_l \in \ch^3(X \times C_l)$ is cohomologically trivial. Since the induced map $\ix([Z_l] + [\delta_l])_{\ast} \colon H^4(X, \mbz) \ra H^2(C_l, \mbz)$ is zero, for each $2 \leqslant i \leqslant \rho$ and $[S_i] \in \ch^2(X)$ in the  orthogonal basis \eqref{def_basis}, we deduce from the definition of $Z_l$ that  $$0 = \ix([Z_l] + [\delta_l])_{\ast}([S_i]) = \ix[\delta_l]_{\ast}[S_i] = \ix [\theta_i] \ \textup{in}\  H^2(C_l, \mbz).$$
	Hence each  $\theta_i \in \ch^1(C_l)$ has degree zero, and $\ix \delta_l$  lies in the image of $$\ch^2(X) \otimes \ch^1(C_l)_{\lhom} \ra \ch^3(X \times C_l)_{\lhom}.$$  This concludes the proof.
\end{proof}

The following result is an analogue of the Abel--Jacobi criterion for constant cycle curves on K3 surfaces proved in \cite[Cor.~5.5, Lem.~5.6]{Huybrechts2014}.

\begin{cor} \label{cri_AJ_Cl}
	Let $X \subset \mbp^5$ be a smooth cubic fourfold,  and  $Y \subset X$ be an   hyperplane section with no triple points. Let $l \subset Y$ be a generic line in the sense of \Cref{existence_1dim}, and denote by $\tcl \ra C_l$  the normalization.  
		
	If the Fano surface $F(Y) \subset F(X)$ is a constant cycle surface of order $N$, then the associated cycle $\ix Z_l \in \ch^3(X \times \tcl)_{\lhom}$ defined in \eqref{def_Zl} satisfies the following:
		\begin{enumerate}[label=(\roman*), font=\normalfont]
			\item $N \cdot \Phi_{\rm tr}(\ix  Z_l) = 0$ in $J_{\rm tr}^5(X \times \tcl)$. \label{tr_AJ_tor}
			\item  \label{CH&AJ_tor}$N \cdot (\ix^2  Z_l) = 0$ in $\ch^3(X \times \tcl)_{\hom}$. In particular, $N \cdot \Phi(\ix^2  Z_l) = 0$ in $J^5(X \times \tcl)$.
		\end{enumerate}
		Here $\Phi_{\rm tr}$ and $\Phi$ denote the (transcendental) Abel--Jacobi maps defined in \eqref{tr_AJ_def} and \eqref{AJ_def}, respectively.
\end{cor}
\begin{proof}
 Suppose that $F(Y) \subset F(X)$ is a constant cycle surface of order $N$. By \Cref{ccs_iff_ccc_general}\ref{ccc_decomp}, the incidence correspondence $\tpl^t \subset X \times \tcl$ satisfies 
 \begin{equation} \label{lc_ccs_decomp}
 	N[\tpl^t] = N \cdot \frac{1}{3}\hx^3 \times [\tcl] + Z  \quad \textup{in } \ch^3(X \times \tcl),
 \end{equation}
 where $Z$ lies in the image of $\ch^2(X) \otimes \ch^1(\tcl) \ra \ch^3(X \times \tcl)$. Let $\{[S_i] \in \ch^2(X) \mid 1 \leqslant i \leqslant \rho\}$ be the orthogonal $\mbq$-basis of $\ch^2(X)_{\mbq}$ chosen in \eqref{def_basis}. 
 Multiplying \eqref{lc_ccs_decomp} by the integer $\ix = \prod_{i =1}^{\rho} [S_i]^2$ and substituting   \eqref{def_Zl}, we obtain	\begin{equation} \label{lc_prod_cycle}
 	N\ix Z_l =   \ix Z - N \sum_{i=1}^{\rho} \frac{\ix}{[S_i]^2}[S_i] \times \theta_i \quad \textup{in } \ch^3(X \times \tcl)_{\lhom}.
 \end{equation}

 Since $\ch^2(X) \simeq H^{2,2}(X,\mathbb{Z})$ by \cite{Bloch-Srinivas83} and \cite{Voisin2007-HodgeConj}, the class $N\ix Z_{l}$ lies in the image of $\ch^2(X) \otimes \ch^1(\tcl)_{\lhom} \to \ch^3(X \times \tcl)_{\lhom}$, and thus  in the kernel of $\Phi_{\rm tr}$. This proves \ref{tr_AJ_tor}.

We now prove \ref{CH&AJ_tor}. Since by construction $\ix \ch^2(X) \subset \oplus_{i=1}^{\rho} \mbz[S_i]$,  we may rewrite \eqref{lc_prod_cycle} as
 \begin{equation} \label{eq_corr_tor}
 	N \cdot (\ix  Z_l) =   \sum_{i = 1}^{\rho} [S_i] \times \alpha_i \ \ \textup{in} \ \ch^3(X \times \tcl)_{\lhom},
 \end{equation}
 where $\alpha_i \in \ch^1(\tcl)_{\lhom}$ for each $i$. This together with the definition of $Z_l$ (see \eqref{def_Zl}) implies
 \begin{equation} 
 	0 = N\cdot (\ix[Z_l])_{\ast}([S_i]) = [S_i]^2 \cdot \alpha_i  \quad \text{in }  \ch^1(\tcl)_{\lhom}, \ \ 1 \leqslant i \leqslant \rho.
 \end{equation} 
 Hence $\ix  \alpha_i = 0$ for each $i$. From \eqref{eq_corr_tor} it follows that  $\ix^2  Z_l \in \ch^3(X \times \tcl)_{\lhom}$ is $N$-torsion, and so is its Abel--Jacobi image $\Phi(\ix^2 Z_l) \in J^5(X \times \tcl)$. This completes the proof.
\end{proof}
\begin{rmk}
For a general Lagrangian surface $S \subset F(X)$ with a fixed resolution $\ts \to S$, 
a direct generalization of the method of \cite[Sec.~5.2]{Huybrechts2014} yields an analogous 
criterion in the intermediate Jacobian $J^5(X \times \ts)$; see \Cref{bridge}. 
However, this approach loses information about the order of constant cycle surfaces; 
see \Cref{problem_coeff} for details.
\end{rmk}
\section{Finiteness of constant cycle Fano surfaces} \label{sec_finiteness}

Let $F(X)$ be the Fano variety of lines of a smooth cubic fourfold $X$. This section is devoted to proving \Cref{main_thm_intro}: for any integer $N > 0$, the set
\begin{equation} \label{orderN_locus}
	\{ Y \in |\mathcal{O}_X(1)| \mid F(Y) \subset F(X) \ \textup{is a constant cycle surface of order} \ N \}
\end{equation}
is finite. We follow the strategy of \cite[Sec.~5]{Huybrechts2014}, applying \Cref{cri_AJ_Cl} to the universal family of hyperplane sections to reduce the problem to the finiteness of the torsion loci of certain normal functions.
In \Cref{toolbox}, we review the theory of normal functions and adapt \cite[Sec.~5.3]{Huybrechts2014} to a general setting; this framework is then applied in \Cref{proof_finiteness} to prove \Cref{main_thm_intro}.

\subsection{Toolbox: Normal functions} \label{toolbox}
The following result is crucial for the proof of \Cref{main_thm_intro}. 
\begin{prop} \label{fini_tor}
	Let $p \colon \mathcal{X} \ra B$ be a smooth projective morphism, with $B$ smooth and quasi-projective. Suppose that $\mcz \in \ch^k(\mathcal{X})$ is a codimension-$k$ class   satisfying the following:
	\begin{enumerate}[label=(\roman*), font=\normalfont]
		\item For every $b \in B$, the cohomology class of the restriction $\mcz_b \coloneqq \mcz|_{\mathcal{X}_b} \in \ch^k(\mathcal{X}_b)$  is trivial;
		\item For any one-dimensional closed subvariety $T \subset B$ and $\mathcal{X}_{T} \coloneqq \mathcal{X} \times_{B} T$, the restriction $$\mcz_{T} \coloneqq  Z|_{\mathcal{X}_{T}} \in \ch^k(\mathcal{X}_{T})$$ is nontrivial and  not  supported on fibres of the restriction $p_T \colon \mathcal{X}_T \ra T$. \label{restriction}
	\end{enumerate} Then for every  integer $N > 0$, the locus $\{ b \in B \mid \mcz_b \in \ch^k(\mathcal{X}_b) \ \textup {is $N$-torsion}\}$ is finite. 
\end{prop}
\begin{rmk}
	The assumption in \Cref{fini_tor}\ref{restriction} that $\mcz_T$ is not supported on fibres is necessary.
	For example, when $p \colon  \mbp^1 \times \mbp^1 \to \mbp^1$ is the first projection, the fibre class  $f \in \ch^1(\mbp^1 \times \mbp^1)$   restricts to zero on every closed fibre $b \times \mbp^1$. However, the zero locus of $f$ is the entire base $\mbp^1$, which is infinite. Note that assumption \ref{restriction} can be alternatively formulated as follows: for any one-dimensional (locally closed) subvariety $T_0 \subset B$, the restriction $\mathcal{Z}_{T_0}$ is nontrivial.
\end{rmk}

\begin{rmk}
	In general, for a flat morphism $p \colon \mathcal{X} \ra B$ with $B$ smooth, 
	the $N$-torsion locus of a cycle $\mcz \in \ch^k(\mathcal{X})$ is known to be  a countable union of proper closed algebraic subsets of $B$; see \cite[Prop.~2.4]{voisin2015unirational}. This alone does not imply \Cref{fini_tor}: under condition \ref{restriction}, the $N$-torsion locus of $\mcz$  has no irreducible components of positive dimension, but may a prior consists of (infinitely) countably many points. 
\end{rmk}
The  proof of \Cref{fini_tor} relies on  the theory of normal functions. We begin by recalling the relevant constructions and properties; for a general reference we refer to  \cite[Ch.~7]{Voisin_HodgeII}.

Let $p \colon \mathcal{X} \ra B$ be a smooth projective morphism with $B$ smooth and quasi-projective, and let $p_J \colon J^{2k-1} \ra B$ be the associated family of intermediate Jacobians. For every $b \in B$, the fibre $p_{J}^{-1}(b) = J^{2k-1}(\mathcal{X}_b)$ is the $k$-th intermediate Jacobian, which receives cohomologically trivial  algebraic cycles of codimension $k$ under the Abel--Jacobi map
$\Phi_b \colon \ch^k(\mathcal{X}_b)_{\lhom} \ra J^{2k-1}(\mathcal{X}_b).$

Suppose $\mathcal{Z} \in \ch^k(\mathcal{X})$ is a cycle that is fibrewise cohomologically trivial, i.e.\ the restriction $\mcz_b \coloneqq \mcz|_{\mathcal{X}_b} \in \ch^k(\mathcal{X}_b)$  has trivial cohomology class.  Then the images of the  cycles $\mathcal{Z}_b$ under the Abel--Jacobi maps $\Phi_b$  form a holomorphic section $\nu \coloneqq \nu_{\mathcal{Z}}$ of $p_J \colon J^{2k-1} \ra B$,  called the \textit{normal function} associated with $\mathcal{Z}$. By construction, we have $\nu_b = \Phi_b(\mathcal{Z}_b)$ for every $b \in B$. 

The zero locus of $\mathcal{Z} \in \ch^k(\mathcal{X})$ is contained in that of its associated normal function $\nu$:
\begin{equation} \label{inclusion_vanishing_loci}
	\{b \in B \mid \mathcal{Z}_b = 0  \ \textup{in}\ \ch^{k}(\mathcal{X}_b)\} \subset Z(\nu) \coloneqq \{b \in B \mid \nu_b = 0  \ \textup{in}\ J^{2k-1}(\mathcal{X}_b)\}.
\end{equation}
By   \cite[Cor.~1.3]{Brosnan_Pearlstein_2013} and \cite[Cor.~4.7]{Schnell2012}\footnote{The cited works establish algebraicity for admissible normal functions (see \cite{Saito96admissible} for the notion of admissibility). Since normal functions arising from families of algebraic cycles are automatically admissible  \cite[Rmk~1.7(ii)]{Saito96admissible}, the theorem applies in our setting.}, the vanishing locus $Z(\nu) \subset B$ is  Zariski closed. In particular, if $\dim(B) = 1$ and $\nu \neq 0$, then $Z(\nu)$ consists of at most finitely many point.

A general strategy to show the nontriviality of the normal function $\nu$  associated with the class $\mathcal{Z} \in \ch^k(\mathcal{X})$ is as follows. 
Let $H_{\mbz}^{2k-1} \coloneqq R^{2k-1}p_{\ast}\mbz/\textup{torsion}$ be the local system on $B$. The locally free $\mso_B$-module $\mathcal{H}^{2k-1} \coloneqq H_{\mbz}^{2k-1} \otimes \mso_B$  admits a Hodge filtration  $\mathcal{H}^{2k-1} = F^0\mathcal{H}^{2k-1} \supset \cdots \supset  F^{2k-1}\mathcal{H}^{2k-1}$. Denote by $\mathcal{J}^{2k-1}$  the sheaf of holomorphic sections of the family of intermediate Jacobians $p_J \colon J^{2k-1} \ra B$. Then $\nu \in H^0(B, \mathcal{J}^{2k-1})$ by construction. We have the following natural short exact sequence
\begin{equation} \label{ses_global_IJ}
	0 \ra H_{\mbz}^{2k-1} \ra \mathcal{H}^{2k-1}/F^k\mathcal{H}^{2k-1} \ra \mathcal{J}^{2k-1} \ra 0,
\end{equation}
and its induced long exact sequence has the boundary map $$\delta \colon H^0(B, \mathcal{J}^{2k-1}) \ra H^1(B, R^{2k-1}p_{\ast}\mbz), \ \ \ \ \nu \mapsto \delta\nu.$$  

The image $\delta\nu$ can be  alternatively described using  the Leray spectral sequence $$E_2^{p,q} = H^{p}(B, R^q\pi_{\ast}\mbz) \Rightarrow E^{p+q} = H^{p+q}(\mathcal{X}, \mbz).$$ Indeed, 
consider the following   boundary map of the Leray spectral sequence: \begin{equation} \label{eq_boundary_Leray}
	d \colon \ker(E^{2k} \ra E_{2}^{0,2k}) \ra E_2^{1,2k-1} = H^1(B, R^{2k-1}\pi_{\ast}\mbz).
\end{equation}  
Since $\mathcal{Z} \in \ch^k(\mathcal{X})$ is fibrewise cohomologically trivial by construction,  its integral cohomology class $[\mathcal{Z}]$ lies in $\ker(E^{2k} \ra E_{2}^{0,2k})$. By \cite[Lem.~8.20]{Voisin_HodgeII}, we have $$\delta\nu =  d[\mathcal{Z}] \in H^1(B, R^{2k-1}\pi_{\ast}\mbz).$$ 
Therefore,  to show that the normal function $\nu$ associated with an algebraic cycle $\mathcal{Z} \in \ch^k(\mathcal{X})$ is nontrivial, it suffices to show that $d\mathcal{Z} \neq 0$. In particular, if $[\mathcal{Z}] \in H^{2k}(\mathcal{X}, \mbz)$ is nontrivial and $d$ is injective, then  $\nu \neq 0$.

The following observation is a straightforward extension of \cite[Lem.~5.8]{Huybrechts2014} to a more general setting.
\begin{lem} \label{spec_bound_inj}
	Let $p \colon \mathcal{X} \ra B$ be a smooth projective morphism with $B$ smooth and quasi-projective. Assume that $\dim(B) = 1$. Then, for every $k > 0$, the boundary map
		$$d \colon \ker(E^{2k} \ra E_{2}^{0,2k}) \ra E_2^{1,2k-1}$$
	is injective (after possibly shrinking $B$ if $B$ is projective).
\end{lem}
\begin{proof}
We follow the argument of  \cite[Lem.~5.8]{Huybrechts2014}.  Consider the following filtration given by the Leray spectral sequence: $$H^{2k}(\mathcal{X}, \mbz) = E^{2k} = E_0^{2k} \supset E_1^{2k}  \supset \cdots \supset E_{2k}^{2k} = 0, \ \ E_i^{2k}/E_{i+1}^{2k} = E_{\infty}^{i, 2k-i}.$$ The kernel of the boundary map $d$ is  $E_2^{2k}$. If $\dim(B) = 1$, then $E_2^{p,q} = E_{\infty}^{p,q} = 0$ for all $p > 2$; thus the $i$-th filtration $E_i^{2k} = 0$ for $i >2$. Consequently, $$\ker(d) = E_2^{2k} = E_{\infty}^{2,2k-1} = E_2^{2,2k-2}/\textup{im}(d_2^{0,2k-1} \colon E_2^{0,2k-1} \ra E_2^{2,2k-2}).$$
	Since $\dim(B) = 1$ and  $R^{2k-2}\pi_{\ast}\mbz$ is a local system on $B$, by shrinking $B$ (only necessary  if $B$ is projective) we obtain $E_2^{2,2k-2} = H^{2}(B, R^{2k-2}\pi_{\ast}\mbz) = 0$. Thus $d$ is injective.
\end{proof}

Now we are ready to show \Cref{fini_tor}.
\begin{proof}[Proof of \Cref{fini_tor}]
	It suffices to prove that the zero locus $\{ b \in B \mid \mcz_b = 0 \ \textup{in} \ \ch^k(\mathcal{X}_b)\}$ is finite; the finiteness of the $N$-torsion locus for every nonzero integer $N$ then follows by replacing $\mcz$ with $N \mcz$.  Let $\nu$ be the normal function associated  with  $\mathcal{Z}$. By \eqref{inclusion_vanishing_loci} and the algebraicity of the zero locus $Z(\nu)$ \cite{Brosnan_Pearlstein_2013, Schnell2012}, it is enough to show that $\dim Z(\nu) = 0$. 
	
	We argue by contradiction. Suppose  $\dim Z(\nu) \geqslant 1$. Let $T_0 \subset Z(\nu)$ be a one-dimensional subvariety. Replace the family $p \colon \mathcal{X} \ra B$, the class $\mathcal{Z}$ and the associated normal function $\nu$ by their restrictions to $T_0$. By assumption, the cohomology class  $[\mathcal{Z}] \neq 0$, so $d [\mathcal{Z}] = \delta \nu \neq 0$ by \Cref{spec_bound_inj}, and thus $\nu \neq 0$, a contradiction to $T_0 \subset Z(\nu)$.  Therefore $\dim Z(\nu) = 0$, completing the proof.
\end{proof}

\subsection{Proof of finiteness of constant cycle Fano surfaces} \label{proof_finiteness}
 Let  $\mcy \ra \mbp \coloneqq |\mso_X(1)|$ be the universal family of  hyperplane sections of the smooth cubic fourfold $X$, with fibre $\mcy_b$ over $b \in \mbp$. 
 Recall from \Cref{cri_AJ_Cl} that, if  $F(\mcy_b)$ is a constant cycle curve of order $N$, then for  a  generic line $l_b \subset \mcy_b$ and the curve $\mcc_b \coloneqq C_{l_b}$ of lines meeting $l_b$, the associated class $\ix^2 Z_{l_b} \in \ch^3(X \times \tmcc_b)$ (see \eqref{def_Zl}) is $N$-torsion.
 
 The strategy to prove \Cref{main_thm_intro} is as follows. We first stratify $\mbp$ so that, over each stratum $B$, one can choose a family of lines $\{l_b \subset \mcy_b \mid b \in B \}$ for which the associated family of curves $\{\mcc_b \mid b \in B\}$ admits a simultaneous resolution $\{\tmcc_b \mid b \in B\}$. We then verify that the family of cycles $\{ \ix Z_{l_b} \in \ch^3(X \times \tmcc_{b})_{\lhom} \mid b \in B\}$ satisfies the hypotheses of \Cref{fini_tor}. Applying  \Cref{fini_tor} to each stratum then yields the desired finiteness.

\subsubsection{Stratification and simultaneous resolution}
Let $U \subset \mbp$ be the locus of hyperplane sections of $X$ with no triple points. Recall from \Cref{exclude_cone} that $\mbp \backslash U$ is  a finite set. As a preliminary step, we stratify $U$ so that on each stratum  \Cref{fini_tor} can be applied. 

Let $F(\mcy/\mbp) \ra \mbp$ denote the relative Fano variety of lines of the universal family $\mcy \ra \mbp$. Its fibre over a point $b \in \mbp$ is the Fano surface  $F(\mcy_b)$ of the cubic threefold $\mcy_b$. For a variety $B$ with a morphism $B \ra \mbp$,  let $\mcy_{B} \coloneqq \mcy \times_{\mbp} B$ denote the base change and $F(\mcy_B/B) \ra B$ its relative Fano variety of lines. 

For convenience we introduce the following notion: 
\begin{defn} \label{def_general_section}
	We call a section $s \colon B \ra F(\mcy_B/B)$ of the family $F(\mcy_B/B) \ra B$ of Fano surfaces \textit{general}, if for every $b \in B$, the line $l_b \subset \mcy_b$ corresponding to the point $s(b) \in F(\mcy_b)$ is generic in the sense of \Cref{existence_1dim}.
\end{defn}

A general section $s$ of $F(\mcy_B/B) \ra B$ then  determines a family $\mcc_B \ra B$ of curves, whose fibre $\mcc_b$ over $b \in B$ is  precisely the curve $C_{l_b} \subset F(\mcy_b)$ of lines intersecting $l_b$ (see \eqref{def_Cl}).

We need the following general result on simultaneous resolution:
\begin{thm}[{\cite[p.~80]{Teissiersimulresol}}] \label{thm_simul_resol} 
	Let  $f \colon \mcc  \ra B$ be a flat family of connected, reduced projective curves over a normal variety $B$. Then $f$ admits a simultaneous resolution given by the normalization $\nu \colon \tmcc \ra \mcc$. Specifically, for each $b \in B$, the restriction  $\nu_b \colon \tmcc_b \ra \mcc_b$ is the normalization, and the composition 
	$f \circ \nu \colon  \tmcc \ra B$
	is a flat family of smooth projective curves.
\end{thm}

We construct a stratification of $U$ as follows.
\begin{lem} \label{stratification}
	The locus $U \subset \mbp$ admits a stratification $U = \coprod_{i \in I} U_i$ into finitely many (locally closed) subvarieties with the following properties for each $i \in I$:
	\begin{enumerate}[label=(\roman*), font=\normalfont]
		\item There exists a smooth variety $V_i$ with a finite  morphism $V_i \ra U_i$ such that the relative Fano variety  $F(\mcy_{V_i}/V_i) \ra V_i$ admits a general section $s_i \colon V_i \ra F(\mcy_{V_i}/V_i)$;
		\item The family of curves $\mcc_{V_i} \ra V_i$ associated with $s_i$ admits a simultaneous resolution via the normalization $\nu \colon \tmcc_{V_i} \ra \mcc_{V_i}$.
	\end{enumerate} 
\end{lem}

\begin{proof}
This is standard;  we only sketch the argument. One first constructs a finite stratification of $U$ satisfying (i); the existence of a general section $s_i$ over each $V_i$ is ensured by \Cref{existence_1dim}. One then refines the stratification so that on each stratum the conditions for the existence of simultaneous resolution in \Cref{thm_simul_resol} are satisfied.
\end{proof}

\subsubsection{Construction of the global cycle} 
Let $i \in I$ and $V_i$ be as in  \Cref{stratification}, and write $B \coloneqq V_i$. Then the relative Fano variety $F(\mcy_{B}/B) \ra B$ admits a general section $s_i$.  For $b \in B$, let $l_b$ denote the line corresponding to the point $s_i(b) \in F(\mcy_b)$. The section $s_i$ determines a family of curves $h \colon \mcc_B \ra B$ whose fibre $\mcc_b \coloneqq h^{-1}(b)$ over $b \in B$ is the curve of lines in $\mcy_b$ meeting  $l_b$ (see \eqref{def_Cl}). By construction, the family $h$ admits a simultaneous resolution $\tmcc_{B} \ra \mcc_B$; both $B$ and $\tmcc_B$ are smooth.  
We have the natural diagram
\begin{equation} \label{diag_pullback}
	\begin{tikzcd}[column sep=3.5em]
		& X \times \tmcc_b  \arrow[r, hook, "\textup{id}_X \times i_b"]
		&X \times \tmcc_{B} \arrow[r,"\textup{id}_X \times p_B"] \arrow[d] &X \times F(X)\\
		& b \arrow[r, phantom, "\in"] &B.
	\end{tikzcd}
\end{equation}

Let $\{[S_i] \in \ch^2(X) \mid 1 \leqslant i \leqslant \rho \}$  be the orthogonal $\mbq$-basis of $\ch^2(X)_{\mbq}$ chosen in \eqref{def_basis}, and let $\lx^t \subset X \times F(X)$ denote the dual Fano correspondence. Set $[T_i] \coloneqq [\lx^t]_{\ast}[S_i] \in \ch^1(F(X))$. Recall that $[S_1] = \hx^2$, so $[T_1] = \gx$ is the Pl\"ucker polarization. 
  
Consider the algebraic cycle 
\begin{equation} \label{def_ZX} 
	Z_X \coloneqq [\lx^t] - \frac{1}{3}\hx^3 \times [F(X)] -  \frac{1}{3}\hx^2 \times \gx - \sum_{i=2}^{\rho} \frac{1}{[S_i]^2}[S_i] \times [T_i] \in \ch^3(X \times F(X))_{\mbq}
\end{equation} 
and its pullbacks via the morphisms in \eqref{diag_pullback}:
\begin{equation}\label{def_ZB}
	Z_B \coloneqq (\textup{id}_X \times p_B)^{\ast}Z_X \in \ch^3(X \times \tmcc_{B})_{\mbq}, \quad Z_b \coloneqq (\textup{id}_X \times i_b)^{\ast}Z_B \in  \ch^3(X \times \tmcc_b)_{\mbq}, \quad b \in B.
\end{equation}
Let $f_b \colon \tmcc_b \ra \mcc_b \hookrightarrow F(X)$ denote the composition.  Write $\widetilde{P}_b \subset \tmcc_b \times X$ for the incidence correspondence and  $\widetilde{P}_b^{t} \subset X \times \tmcc_b$ for its transpose. Set  $\theta_{b,i} \coloneqq [\widetilde{P}_b^t]_{\ast}[S_i] \in \ch^1(\tmcc_b)$.  Then
\begin{equation*}
	\delta_b \coloneqq \sum_{i=2}^{\rho} \frac{1}{[S_i]^2}[S_i] \times \theta_{b,i}, \quad Z_b = [\widetilde{P}_b^{t}] - \frac{1}{3}\hx^3 \times [\tmcc_b]  - \frac{1}{3}\hx^2 \times f_b^{\ast}(\gx) - \delta_b  \in \ch^3(X \times \tmcc_b)_{\mbq}
\end{equation*}
 are exactly the classes $\delta_{l_b}, Z_{l_b}$ in the notation of \eqref{def_Zl}. As before, the integer $\ix \coloneqq \prod_{i=1}^{\rho}[S_i]^2$  satisfies $\ix  Z_X \in \ch^3(X \times F(X))$, $\ix Z_B \in \ch^3(X \times \tmcc_B)$, and $\ix Z_b \in \ch^3(X \times \tmcc_b)$.  

The following lemma shows that the  cohomology class $\ix[Z_{B}] \in H^6(X \times \tmcc_{B}, \mbz)$ is fibrewise  trivial over $B$ yet globally nontrivial.

\begin{lem} \label{global_cycle_property}
	Let $Z_{B} \in \ch^3(X \times \tmcc_{B})_{\mbq}$ be as defined in \eqref{def_ZB}. Then: 
	\begin{enumerate}[label=(\roman*), font=\normalfont]
		\item For any $b \in B$, the cycle $\ix Z_b \in \ch^3(X \times \tmcc_b)$ is cohomologically trivial. 
		\item For any one-dimensional  subvariety $T_0 \subset B$ and $\tmcc_{T_0} \coloneqq \tmcc_B |_{T_0}$, the restriction  $$\ix  Z_{T_0} \coloneqq \ix  Z_{B}|_{X \times \tmcc_{T_0}} \in \ch^3(X \times \tmcc_{T_0})$$ is cohomologically nontrivial.
	\end{enumerate}   
\end{lem}
\begin{proof} 
	The first claim follows directly from \Cref{decomp_Pl}. We now prove (ii).	
	
	Let $T_0 \subset B$ be a one-dimensional subvariety. By construction, the family $\tmcc_{T_0} \ra T_0$ of curves is the simultaneous resolution of  $h_{T_0} \colon \mcc_{T_0} \ra T_0$, and the fibre $\mcc_b = h_{T_0}^{-1}(b)$ over $b \in T_0$ is the curve $C_{l_b} \subset F(\mcy_b)$ of lines on $\mcy_b$ that meets $l_b$.  Choose smooth compactifications $T_0 \subset T$ and $\tmcc_{T_0} \subset \tmcc_{T}$ so that the natural projection $p_{T_0} \colon \tmcc_{T_0} \ra F(X)$ extends to a proper morphism $p_{T} \colon \tmcc_{T} \ra F(X)$.
	Let $\ix Z_{T} \in \ch^3(X \times \tmcc_{T})$ be the pullback of $\ix Z_X$ (see \eqref{def_ZX}) via $\textup{id}_X \times p_{T} \colon X \times \tmcc_T \ra X \times F(X)$. The induced map $\ix [Z_T]_{\ast}$ in degree zero factors as
	\begin{equation*}
		\ix[Z_{T}]_{\ast} \colon H^0(\tmcc_{T},\mbz) \xrightarrow{(p_T)_{\ast}} H^4(F(X), \mbz) \xrightarrow{[\lx]_{\ast}} H^2(X, \mbz).
	\end{equation*}
	By construction, $\ix Z_T$ restricts to $\ix Z_{T_0}$ on $X \times \tmcc_{T_0}$. To see $\ix[Z_{T_0}] \in H^6(X \times \tmcc_{T_0}, \mbz)$ is nontrivial, it suffices to prove that $\ix[Z_{T}]_{\ast}[\tmcc_{T}] \neq 0$. 
	
	First, we verify that $(p_T)_{\ast}[\mathcal{C}_T] \neq 0$, i.e.\  $\text{im}(p_T)$ is a surface on $F(X)$. 	
	Suppose to the contrary that the family of curves $\mathcal{C}_{T_0} \to T_0$  is constant on $F(X)$.  Since the family $\mathcal{Y}_{T_0} \to T_0$ of hyperplane sections of $X$ is non-constant by \Cref{stratification},  
	it follows from \Cref{criterion_ccc_iff_ccs} and \Cref{equi_one_cycle} that the threefold 
	$D_F \coloneqq \{ [l] \in F(\mathcal{Y}_t) \mid t \in T_0 \} \subset F(X)$ 
	would be supported on a single curve $\mathcal{C}_b \subset F(X)$ up to rational equivalence of zero-cycles.
	By \cite[Cor.~1.2]{Voisin2016}, the orbit of a point on $F(X)$ under rational equivalence is a countable union of closed algebraic subsets of dimension at most two, so $D_F$ would be covered by a one-dimensional family of constant cycle surfaces. 
	Since every point on a constant cycle surface represents the canonical class $o_F \in \ch_0(F(X))$, the divisor $D_F$ would then be a constant cycle subvariety on $F(X)$, which is impossible. 
	Hence $\mathcal{C}_{T_0} \to T_0$ is nonconstant.

	We now show that $[\lx]_{\ast}[\textup{im}(p_T)] \neq 0$. Let $D \subset X$ be the locus swept out by the family of lines parametrized by the surface $\textup{im}(p_T)$. 
	If $\dim D = 3$, then $[\lx]_{\ast}[\textup{im}(p_T)] = n \cdot \hx$ for some nonzero integer $n$, and we are done. 
	If $\dim D = 2$, then $D$ would be a union of planes contained in $X$. 
	This would imply that $\dim \mcc_b = 2$ for all $b \in B$, contradicting the generality of the line $l_b \subset \mcy_b$ ensured by \Cref{stratification}. Thus $\ix[Z_{T}]_{\ast}[\tmcc_{T}] \neq 0$, completing the proof.
\end{proof}

\subsubsection{Proof of \Cref{main_thm_intro}} 
Let $U \subset \mbp = |\mso_X(1)|$ denote the locus of hyperplane sections without triple points, and let $U = \coprod_{i \in I} U_i$ be the finite stratification given by \Cref{stratification}. Fix a stratum $U_i$, and write $B \coloneqq V_i$ for the smooth  variety mapping onto $U_i$ with the properties stated in \Cref{stratification}. Denote by $\mcc_B \ra B$ the associated family of curves, and by $\tmcc_B \ra \mcc_B$  the simultaneous resolution, whose fibre over $b \in B$ is the normalization $\tmcc_b \ra \mcc_b$. 
By \Cref{cri_AJ_Cl}\ref{CH&AJ_tor}, for every integer $N>0$, 
$$\{b \in B \mid F(\mcy_b)\ \textup{is a constant cycle surface of order N} \} \subset \{b \in B \mid \ix^2 Z_b \in \ch^3(X \times \tmcc_b)[N]\}.$$
 Since the cycle $\ix^2  Z_B \in \ch^3(X \times \tmcc_B)$, viewed as a family of cycles over $B$, satisfies the  assumptions of \Cref{fini_tor} by \Cref{global_cycle_property}, the latter set  is  finite for every $N$, and hence so is the former. Finally, as the stratification is finite and $\mbp \setminus U$ is a finite set by \Cref{exclude_cone}, summing over all strata yields the desired finiteness. \qed

\section{Appendix: An alternative criterion for constant cycle surfaces} \label{appendix}
We present an alternative criterion for constant cycle surfaces on the Fano variety $F(X)$ of a smooth cubic fourfold $X$, applicable to all Lagrangian surfaces on $F(X)$ rather than just to Fano surfaces $F(Y)$ of hyperplane sections $Y \subset X$. However, this approach does not detect the order of a constant cycle surface, and therefore cannot be used to prove \Cref{main_thm_intro}.
\subsection{Chow-theoretic decomposition} \label{app_chow_cri} 
We fix the following notation. Let $\hx \in \ch^1(X)$ denote the hyperplane class,  $\gx \in \ch^1(F(X))$ the Plücker polarization, and $c_2 \coloneqq c_2(\mcs_F)$  the second Chern class of the universal subbundle $\mcs_F$ on $F(X)$.

For an integral surface $S  \subset F(X)$,  fix  a  resolution of singularities  $\ts \ra S$, and let $f \colon \ts \ra S \hookrightarrow F(X)$ be the composition with graph $\Gamma_f \subset \ts \times F(X)$. Let $\widetilde{P} \subset \ts  \times X$ and $\lx \subset F(X) \times X$ denote the incidence correspondences. Note that  $[\widetilde{P}] = [\lx] \circ [\Gamma_f]$ as correspondences.

The following observation refines the Chow-theoretic criterion   \Cref{equi_def_ccs}\ref{BS_decomp}, using the vanishing $\ch^2(X)_{\lhom} = 0$   \cite[Thm.~1]{Bloch-Srinivas83}.

\begin{lem} \label{lift_Fano_decomp_chow}
	Let $S \subset F(X)$ be an integral surface, and $f \colon \ts \ra F(X)$ be as above. Then $S$ is a constant cycle surface if and only if there exists an integer $N > 0$, such that
	\begin{equation*}
		N[\widetilde{P}^t] = N \cdot \frac{1}{3}\hx^3 \times [\ts] + \alpha_{2,1} + \alpha_{1,2} \in \ch^3(X \times \ts), 
	\end{equation*}
	where $\alpha_{2,1} \in \ch^2(X) \otimes \ch^1(\ts)$ and
	$\alpha_{1,2} = N \cdot \hx \times f^{\ast}(\frac{1}{3}(\gx^2 - c_2)) \in \ch^1(X) \otimes \ch^2(\ts)$.
\end{lem} 
\begin{proof}
	The ``if" direction follows from the implication \ref{BS_decomp} $\Rightarrow$ \ref{ccs_def} in \Cref{equi_def_ccs}.
	Conversely, let $S \subset F(X)$ be a constant cycle surface of order $N^\prime$. Then  \Cref{equi_def_ccs} and \Cref{ord_alt_def} imply
	\begin{equation} \label{eq_BS_dec_step0}
		N^{\prime}[\Gamma_f^t] = N^{\prime} \cdot o_F \times [\ts] + Z^{\prime} \ \ \textup{in} \ \ch^4(F(X) \times \ts),
	\end{equation} where  $Z^{\prime}$ is a cycle supported on $F(X) \times D$ for some proper subset $D \subset \ts$. Composing \eqref{eq_BS_dec_step0} with $[\lx^t] \in \ch^3(X \times F(X))$ and using  $[\lx]_{\ast}(o_F) = \frac{1}{3}\hx^3 \in \ch_1(X)$ (see \Cref{dist_line_class}), we obtain  
	\begin{equation} \label{eq_BS_dec_step1}
		N^{\prime}[\widetilde{P}^t] = N^{\prime} \cdot \frac{1}{3}\hx^3 \times [\ts] + \alpha \in \ch^3(X \times \ts),
	\end{equation}
	where $\alpha$ is supported on $X \times D$ for some proper closed subset $D \subset \ts$.

	After multiplying by a suitable nonzero integer, the cycle $\alpha$ admits a further decomposition into product cycles. To see this, write $\alpha = \sum_{i} n_i[Z_i]$ as a sum of cycles of proper  subvarieties $Z_i \subset X \times \ts$, where each $Z_i$ is supported on $X \times D_i$ for some irreducible curve $D_i \subset \ts$. If the image of $Z_i \subset X \times \ts \ra \ts$ is a finite set of points, then  $[Z_i]$ is  a product cycle. Otherwise, the subvariety $Z_i \subset X \times D_i$ dominates $D_i$. Let $\nu_i \colon \widetilde{D}_i \ra D_i$ be the normalization, and denote by $\widetilde{Z}_i \subset X \times \widetilde{D}_i$ the pullback of $Z_i$.  The cycles in the family $\{ [\widetilde{Z}_{i,t}] \in \ch^2(X) \mid t \in \widetilde{D}_i\}$ are homologically equivalent. Since $\ch^2(X)_{\lhom} = 0$  \cite[Thm.~1]{Bloch-Srinivas83}, these cycles are constant. Hence, the image of $[\widetilde{Z}_{i}]_{\ast} \colon \ch^1(\widetilde{D}_i) \ra \ch^2(X)$ is generated by some class $\theta_i \in \ch^2(X)$. The Bloch--Srinivas principle applied to the correspondence $\widetilde{Z}_i \subset X \times \widetilde{D}_i$ yields 
	\begin{equation} \label{eq_BS_dec_tilde}
		N_i [\widetilde{Z}_i] = N_i \cdot \theta_i \times [\widetilde{D}_i] + \widetilde{\beta}_i \in \ch^2(X \times \widetilde{D}_i),
	\end{equation}
	where $N_i$ is a nonzero integer and $\widetilde{\beta}_i$ lies in the image of $\ch^1(X) \otimes \ch^1(\widetilde{D}_i) \ra \ch^2(X \times \widetilde{D}_i)$. Pushing forward \eqref{eq_BS_dec_tilde} along the morphism
	$X \times \widetilde{D}_i \xrightarrow{\textup{id}_{X} \times \nu_i} X \times D_i \subset X \times \ts$
	gives	\begin{equation}\label{eq_BS_dec_step2}
		N_i [Z_i] = N_i \cdot \theta_i \times [D_i] + \beta_i \in \ch^3(X \times \ts),
	\end{equation}
	where $\beta_i$ lies in the image of $\ch^1(X) \otimes \ch^1(D_i) \ra \ch^3(X \times \ts)$. Multiplying  \eqref{eq_BS_dec_step1} by the product of the integers $N_i$ and substituting \eqref{eq_BS_dec_step2} for each $i$, we obtain a  refined decomposition
	\begin{equation*}
		N[\widetilde{P}^t] = N \cdot \frac{1}{3}\hx^3 \times [\ts] + \alpha_{2,1} + \alpha_{1,2} \in \ch^3(X \times \ts), 
	\end{equation*}
	where $\alpha_{2,1}$ lies in the image of $\ch^2(X) \otimes \ch^1(\ts) \ra \ch^3(X \times \ts)$, and
	$\alpha_{1,2}$ lies in the image of $\ch^1(X) \otimes \ch^2(\ts) \ra \ch^3(X \times \ts)$. By construction, the integer $N$ is divisible by the order of the constant cycle surface $S \subset F(X)$.
	
	It remains to determine the class $\alpha_{1,2}$. Since $\ch^1(X) \simeq \mbz \cdot  \hx$,  one can write $\alpha_{1,2} = \hx \times \beta$ for some $\beta \in \ch^2(\ts)$. Using the factorization 
	\begin{equation*}
		[\widetilde{P}^t]_{\ast} \colon \ch_1(X) \xrightarrow{[\lx^t]_{\ast}} \ch_2(F(X)) \xrightarrow{[\Gamma_f^t]_{\ast}} \ch_0(\ts)
	\end{equation*} and the identity  $[\lx^t]_{\ast}(\frac{1}{3}\hx^3) = \frac{1}{3}(\gx^2 - c_2)$  (see \cite[Lem.~A.5]{ShenVial2013FTofHK}),  we find 
	\begin{equation*}
		\beta  =  N[\widetilde{P}^t]_{\ast}(\frac{1}{3}\hx^3) = N \cdot  f^{\ast}(\frac{1}{3}(\gx^2 - c_2)) \in \ch^2(\ts).
	\end{equation*} This concludes the proof.
\end{proof}

\begin{rmk} \label{problem_coeff}
It seems that the minimal integer $N$, for which a decomposition into product cycles as above exists, cannot be bounded above by a constant multiple of the order of the constant cycle surface, even when $S$ is the Fano surface of a hyperplane section $Y \subset X$. 
\end{rmk}

\subsection{Abel--Jacobi criterion} \label{aj_map}
Let $S \subset F(X)$ be an integral surface with a fixed resolution of singularities $\ts \ra S$. In analogy with \Cref{AJ_Cl}, we construct a cycle $\zs \in \ch^3(X \times \ts)_{\lhom}$ by modifying the correspondence  $[\widetilde{P}^t]$, and derive a criterion for $S$ to be a constant cycle surface in terms of  the Abel--Jacobi image of $\zs$.

We fix the setup. Consider the Abel--Jacobi map
\begin{equation} \label{aj_def_S}
\Phi \colon \ch^3(X \times \ts)_{\lhom} \ra J^5(X \times \ts) = \frac{\big(F^4H^7(X \times \ts, \mbc)\big)^{\ast}}{H_7(X \times \ts, \mbz)}. 
\end{equation}
Since $X$ has no nontrivial odd cohomology and $H^6(X,\mbz) \simeq \mbz$, we have
\begin{align*}
J^5(X \times \ts)  
&\simeq \frac{\big(F^4(H^6(X,\mbc) \otimes H^1(\ts,\mbc))\big)^{\ast}}{H_6(X,\mbz)\otimes H_1(\ts, \mbz)}
\oplus \frac{\big(F^4(H^4(X,\mbc) \otimes H^3(\ts,\mbc))\big)^{\ast}}{H_4(X,\mbz)\otimes H_3(\ts, \mbz)}\\
& \simeq \textup{Alb}(\ts) \oplus \frac{(H^{2,2}(X) \otimes H^{2,1}(\ts) \oplus H^{3,1}(X) \otimes H^3(\ts, \mbc))^{\ast}}{H_4(X, \mbz) \otimes H_3(\ts, \mbz)}.
\end{align*}
Let $\rho(X)$ denote the rank of $H^{2,2}(X,\mbz)$, and let $H_{\rm{tr}}^4(X,\mbz)$ and   $H_{\rm{tr}}^{2,2}(X)$ be as in \eqref{eq_tr_coh}. Then the algebraic part and the transcendental part of $J^5(X \times \ts)$ are given respectively by
\begin{align*}
	J_{\rm{alg}}^5(X \times \ts) &\simeq (H^{3,3}(X,\mbz) \otimes \textup{Alb}(\ts)) \oplus (H^{2,2}(X,\mbz) \otimes J(\ts)) 
	\simeq \textup{Alb}(\ts) \oplus J(\ts)^{\oplus \rho(X)},\\
	J_{\rm{tr}}^5(X \times \ts) &\coloneqq \frac{J^5(X \times \ts)}{J_{\rm{alg}}^5(X \times \ts)} = \frac{(H^{3,1}(X) \otimes H^3(\ts,\mbz) \oplus H_{\rm{tr}}^{2,2}(X) \otimes H^{2,1}(\ts))^{\ast}}{H_{\rm{tr}}^4(X,\mbz) \otimes H^1(\ts, \mbz)}.
\end{align*}

Since algebraic and homological equivalence coincide for  cycles on a smooth surface,  cycles in the image of the natural map
\begin{equation*}
	G_{\lhom} \coloneqq \ch^2(X) \otimes \ch^1(\ts)_{\lhom} \oplus \ch^1(X) \otimes \ch^2(\ts)_{\lhom} \ra  \ch^3(X \times \ts)
\end{equation*} are algebraically equivalent to zero, and hence their image under  $\Phi$ is contained in  $J_{\rm{alg}}^5(X \times \ts)$. 
We have the following commutative diagram with exact rows:
\begin{equation}
\begin{tikzcd}
& &G_{\lhom} \arrow[r] \arrow[d,"\Phi_{\textup{alg}}"] &\ch^3(X \times \ts)_{\lhom} \arrow[r] \arrow[d,"\Phi"] &\frac{\ch^3(X \times \ts)_{\lhom}}{G_{\lhom}} \arrow[r] \arrow[d] &0\\
&0 \arrow[r] &J_{\textup{alg}}^5(X \times \ts) \arrow[r] &J^5(X \times \ts) \arrow[r, "p_{\textup{tr}}"] &J_{\textup{tr}}^5(X \times \ts) \arrow[r] & 0.
\end{tikzcd}
\end{equation}
By construction, the  induced map $\Phi_{\rm{alg}}$ is essentially the direct sum of the Albanese map $\Phi_{\ts} \colon \ch^2(\ts)_{\lhom} \ra \textup{Alb}(\ts)$ and $\rho(X)$ copies of the isomorphism $\ch^1(\ts)_{\lhom} \simeq J(\ts)$ given by the Abel--Jacobi map.  As in \Cref{AJ_Cl}, we call 
 \begin{equation} \label{tr_AJ_def_S}
 	\Phi_{\rm{tr}} \coloneqq p_{\rm{tr}} \circ \Phi \colon \ch^3(X \times \ts)_{\rm{hom}} \ra J_{\rm{tr}}^5(X \times \ts)
 \end{equation} the \textit{transcendental Abel--Jacobi map}. 

We now  modify  $[\widetilde{P}^t]$ by product cycles to obtain a cohomologically trivial cycle.  Let 
\begin{equation}\label{Q_basis}
	\left\{[S_i] \in \ch^2(X) \mid 1 \leqslant i \leqslant \rho \right\}
\end{equation} be an orthogonal $\mbq$-basis of $\ch^2(X)_{\mbq}$ as in \eqref{def_basis}.  Then $\oplus_{i=1}^{\rho} \mbz [S_i] \subset \ch^2(X) \simeq H^{2,2}(X,\mbz)$ is a sublattice of finite index. Write $[C_i] \coloneqq [\widetilde{P}^t]_{\ast}[S_i] \in \ch^1(\ts)$. 
Consider the following   cycle 
\begin{equation} \label{modified_corr}
\zs \coloneqq [\widetilde{P}^t] - \frac{1}{3}\hx^3 \times [\ts] - \sum_{i} \frac{1}{[S_i]^2}[S_i] \times [C_i]-  \hx \times \frac{1}{3} f^{\ast}(\gx^2 - c_2) \in \ch^3(X \times \ts)_{\mbq},
\end{equation} 
where $f \colon \ts \ra F(X)$ is the natural morphism.
The integer $\ix \coloneqq \prod_{i=1}^{\rho}[S_i]^2$ is independent of the surface $S \subset F(X)$. By  construction, $\ix \zs  \in \ch^3(X \times \ts)$. 

Note that when $S = F(Y)$ is the Fano surface of a smooth hyperplane section $Y \subset X$, and $l \subset Y$ is a non-special line (so that the curve $C_l$ is also  smooth),  the class $\ix\zs \in \ch^3(X \times S)$ pulls back to $\ix Z_l \in \ch^3(X \times C_l)_{\lhom}$ (see \eqref{def_Zl}) via the closed embedding $X \times C_l \subset X \times F(Y)$. 

As in  \Cref{decomp_Pl}, we have the following:

\begin{lem} \label{construct_coh_0}
Suppose that the surface $S \subset F(X)$ is Lagrangian. Then the  cycle $\ix^2 \zs \in \ch^3(X \times \ts)$ is cohomologically trivial.
\end{lem}
\begin{proof}
Note that for each $0 \leqslant k \leqslant 4$, the group $H^{2k}(X,\mbz)$ is torsion-free and unimodular, i.e.\ $H^{2k}(X, \mbz) \simeq H^{8-2k}(X, \mbz)^{\vee}$ (see e.g.\ \cite[Sec.~1.1.5]{Huybrechts_cubicbook}). 
By the Künneth decomposition, it therefore suffices to show  that  $\ix^2[\zs]_{\ast} \colon H^{2k}(X,\mbz) \ra H^{2k-2}(\ts,\mbz)$ vanishes for $1 \leqslant k \leqslant 3$. As in \Cref{decomp_Pl}, we compute
\begin{equation*}
[\widetilde{P}^t]_{\ast} \colon	H^{2k}(X, \mbz) \xrightarrow{[\lx^t]_{\ast}}  H^{2k-2}(F(X), \mbz) \xrightarrow{f^{\ast}} H^{2k-2}(\ts, \mbz), \ k = 1,2,3.
\end{equation*}

For $k = 1$,  since $[\widetilde{P}]_{\ast}(\hx) = f^{\ast}[F(X)] = [\ts] \in H^0(\ts,\mbz)$, the map $\ix[\zs]_{\ast}$ vanishes on $H^2(X,\mathbb{Z}) \simeq \mbz \cdot \hx$.

For $k = 2$, the map $[\lx^t]_{\ast}$ induces an isomorphism $H^{3,1}(X) \simeq H^{2,0}(F(X))$.  Since $S \subset F(X)$ is  Lagrangian, the induced map $\ix[\zs]_{\ast}$ (on $H^4(X, \mbc)$) vanishes on $H^{3,1}(X)$, and hence on the transcendental lattice $T(X)$, i.e.\ the smallest saturated sub-Hodge structure of $H^4(X,\mbz)$ whose complexification contains $H^{3,1}(X)$. 
We now show  that $\ix^2[\zs]_{\ast}$ also vanishes on ${H^{2,2}(X,\mbz)}$. By construction, for each class $[S_i] \in \ch^2(X) \simeq H^{2,2}(X,\mbz)$ in the $\mbq$-basis of $\ch^2(X)_{\mbq}$ chosen in \eqref{Q_basis}, one has $\ix[\zs]_{\ast}[S_i] = 0$ in $H^2(\ts,\mbz)$. Since $\ix H^{2,2}(X,\mbz) \subset \oplus_{i=1}^{\rho} \mbz [S_i]$, it follows that  $\ix^2[\zs]_{\ast}$  vanishes on $H^{2,2}(X,\mbz)$. After replacing  $\ix$ by a suitable multiple that is independent of the surface $S$, e.g.\  the index of the sublattice $T(X) \oplus H^{2,2}(X,\mbz) \subset H^{4}(X,\mbz)$, the map $\ix^2[\zs]_{\ast}$ is zero on $H^4(X, \mbz)$.

Finally, for $k = 3$, we have $H^6(X, \mbz) \simeq \mbz \cdot \frac{1}{3}\hx^3$. Since $[\widetilde{P}^t]_{\ast}(\frac{1}{3}\hx^3) = f^{\ast}(\frac{1}{3} (\gx^2 - c_2))$ (see \cite[Lem.~A.5]{ShenVial2013FTofHK}), it follows that the map $\ix[\zs]_{\ast}$ is  trivial on $H^6(X,\mbz)$. This completes the proof.
\end{proof}

\begin{rmk}
The Lagrangian assumption is necessary to ensure that $\ix[\zs]_{\ast}$ is trivial on the transcendental lattice $T(X)$. This assumption is satisfied, for example, when $S$ is the Fano surface of a cubic threefold $Y \subset X$  \cite[Ex.~3(7)]{Voisin_1992}, or when $S$ is a constant cycle surface (see \Cref{ccs_is_Lag}).
\end{rmk}

From \Cref{lift_Fano_decomp_chow} and the discussion above, we obtain the following analogue of  \Cref{cri_AJ_Cl}.
\begin{cor} \label{bridge}
Let $S \subset F(X)$ be an integral Lagrangian surface with a fixed resolution of singularities $\ts \ra S$, and let $\ix^2 \zs \in \ch^3(X \times \ts)_{\lhom}$ be the associated cycle as in \eqref{modified_corr}.

If $S$ is a constant cycle surface on $F(X)$, then $\ix^2\zs$ is torsion. In particular, both $\Phi(\ix^2\zs) \in J^5(X \times \ts)$  and $\Phi_{\rm tr}(\ix^2\zs) \in J_{\rm tr}^5(X \times \ts)$ are  torsion (see \eqref{aj_def_S} and \eqref{tr_AJ_def_S}). 
\end{cor}
\begin{proof} 
By \Cref{lift_Fano_decomp_chow}, there exists an integer $N$ such that	\begin{equation} \label{ccs_dec_1}
N[\widetilde{P}^t] = N \cdot \frac{1}{3}\hx^3 \times [\ts] +   \alpha_{2,1} + N  \cdot \hx \times f^{\ast}(\frac{1}{3}(\gx^2 - c_2)) \ \textup{in} \  \ch^3(X \times \ts), 
\end{equation}
where $\alpha_{2,1}$ lies in the image of $\ch^2(X) \otimes \ch^1(\ts) \ra \ch^3(X \times \ts)$.  
Multiplying \eqref{ccs_dec_1} by the integer $\ix = \prod_{i =1}^{\rho} [S_i]^2$ and applying \eqref{modified_corr} yields 
\begin{equation} \label{eq_Ntor}
N \cdot (\ix \zs) = \ix  \cdot \alpha_{2,1} - N \sum_{i}  \frac{\ix}{[S_i]^2}[S_i] \times [C_i]  \ \textup{in} \ \ch^3(X \times \ts).
\end{equation}

Since $\ix \ch^2(X) \subset \oplus_{i=1}^{\rho} \mbz [S_i]$,  we may write $N\ix  \zs =   \sum_{i = 1}^{\rho} [S_i] \times \beta_i \in \ch^3(X \times \ts)$ with $\beta_i \in \ch^1(\ts)$. By construction (see \eqref{modified_corr}) we have  $$0 = N \ix [\zs]_{\ast}[S_i] = [S_i]^2 \cdot  \beta_i  \in \ch^1(\ts).$$ Therefore, each $\beta_i$ is $\ix$-torsion, and $N  \ix^2 \zs = 0$ in $\ch^3(X \times \ts)$. Consequently, both $\Phi(\ix^2 \zs) \in J^5(X \times \ts)$ and its projection to $J_{\rm tr}^5(X \times \ts)$ are $N$-torsion.
\end{proof}

\begin{rmk} \label{problem_app} 
Unlike \Cref{cri_AJ_Cl}, when $S$ is a constant cycle surface, it is unclear whether  $\Phi(\ix^2\zs) \in J^5(X \times \ts)$  is annihilated by $\ord(S)$ up to a constant multiple independent of $S$; see \Cref{problem_coeff}. As $S$ varies in the family of Fano surfaces $F(Y) \subset F(X)$, the cycles $\ix^2\zs$ still define a normal function $\nu$. However, there may be no nonzero integer $M$, independent of $S$, for which the locus of constant cycle Fano surfaces of order $N$ is  contained in the torsion locus $Z(N M \cdot \nu)$. This explains why the strategy of \Cref{fin_ccc} cannot be applied to  \Cref{main_thm_intro} via \Cref{bridge}.
\end{rmk}

\medskip 

\printbibliography

{\itshape
	\noindent
	\textsc{Mathematisches Institut, Universit\"at Bonn, Endenicher Allee 60, 53115 Bonn, Germany.}\\
	E-mail address: \href{mailto:jxhuang@math.uni-bonn.de}{jxhuang@math.uni-bonn.de}
}
\end{document}